\RequirePackage{fix-cm}
\documentclass[12pt]{amsart}

\usepackage{amsmath, amssymb}    
\usepackage{tikz}
\usepackage{caption}             
\usepackage{graphicx}            
\usepackage{xcolor}
\usepackage{lipsum}
\usepackage{amsthm}   
\usepackage{hyperref}            
\usepackage[all]{hypcap} 
\usepackage{cite}                
\usepackage{geometry}            
\usepackage{enumitem}
\usetikzlibrary{calc, shapes.geometric, shadows, decorations.pathreplacing}

\newtheorem{theorem}{Theorem}[section]

\newtheorem{lemma}[theorem]{Lemma}
\newtheorem{proposition}[theorem]{Proposition}
\newtheorem{corollary}[theorem]{Corollary}
\newtheorem{remark}[theorem]{Remark}
\newtheorem{observation}[theorem]{Observation}
\newtheorem{notation}[theorem]{Notation}
\newtheorem{example}[theorem]{Example}
\newtheorem{definition}[theorem]{Definition}

\renewenvironment{proof}[1][Proof]{\par\noindent\textbf{#1.} }{\hfill$\square$\par}
\newcommand{\customref}[2]{\hyperref[#2]{#1~\ref*{#2}}}

\title{The Point Spectrum of Periodic Quantum Trees}
\author[J.~Breuer and N.~Y.~Levi]{Jonathan Breuer$^{1,3}$ and Netanel Y.~Levi$^{2,3}$}
\thanks{$^1$ Institute of Mathematics, The Hebrew University of Jerusalem, Jerusalem, 91904, Israel. Supported by the Israel Institute for Advanced Studies. E-mail: jbreuer@math.huji.ac.il.}
\thanks{$^2$ Institute of Mathematics, The Hebrew University of Jerusalem, Jerusalem, 91904, Israel. E-mail: netanel.levi2@mail.huji.ac.il.}
\thanks{$^3$ Research supported in part by Israel Science Foundation Grant No. 1378/20, and by the United States-Israel Binational Science Foundation Grant No. 2020027.}
\date{\today}

\begin{document}

\sloppy

\begin{abstract}
We study the point spectrum of a periodic quantum tree equipped with a Schr\"odinger type differential operator with delta-type vertex conditions, using subsets of the compact graph that generates the tree. We prove analogs of existing discrete results concerning the eigenvalues of such operators (see \cite{About_Regular_Graphs,Garza-Vargas}). In particular, we define the density of states measure and find the measure of eigenvalues of the periodic tree.

While most results carry over from the discrete case, a notable difference between the continuum and discrete cases is that a \textbf{regular} quantum periodic tree may have eigenvalues. 

We prove that after an arbitrarily small adjustment of edge lengths, the point spectrum of the universal cover of a compact quantum graph, with at least one cycle and the standard Schr\"odinger operator, is empty. 
\end{abstract}

\maketitle
\section{INTRODUCTION}

Periodic Jacobi matrices on trees have attracted great deal of interest in recent years (\cite{ahf,About_Regular_Graphs,abks,abs,bbgvss,Garza-Vargas,CSZ,ftp,fts,GVK,KLW,nw,sunada,woess} is a partial list). These are discrete Schr\"odinger type operators defined on universal covers of finite connected graphs (a precise definition is given below). Properties of interest include the structure of the spectrum, properties of spectral measures, the density of states, the existence of eigenvalues and the structure of eigenfunctions.

Continuum (aka `metric') graphs arise when considering a graph as a collection of line segments attached via nodes, and are important both for applied and theoretical reasons. When a continuum Schr\"odinger operator is associated to such a metric graph, the model is called a \textit{quantum graph}. Such graphs find applications in chemistry and physics, but are also interesting for their sheer mathematical properties (see, e.g.\ \cite{Introduction_to_Quantum_Graphs} and references therein). 

Periodic quantum trees are the natural continuum analogs of periodic Jacobi matrices on trees that, to the best of our knowledge, have not yet received the same amount of attention as their discrete analogs (the papers \cite{AISW2, Carlson} are notable exceptions).
In this paper we discuss the eigenvalues of a periodic quantum tree (which we will denote by $\sigma_{\mathfrak{p}}\left(\mathcal{H}_{\mathcal{T}}\right)$ and call the \textit{point spectrum}), and the edges and vertices on which the eigenfunctions are supported. Much of what we do is an adaptation of the work in \cite{Garza-Vargas} to the continuum case. There are, however, notable differences in the results.
We start with some definitions.

A \textit{discrete graph} is a pair $G=\left(V,E\right)$ where $V$ is a set of \textit{vertices}, and $E$ is a set of pairs of vertices called the \textit{edges}. For any subset of vertices and edges $X\subseteq G$, we let $V\left(X\right)$ be the set of vertices of $X$ and $E\left(X\right)$ be the set of edges of $X$. For $\upsilon,u\in V$ we write $\upsilon\sim u$ if $\upsilon$ and $u$ are adjacent, that is, there is an edge $\left(\upsilon,u\right)\in E$. Similarly, two edges are adjacent if they share a common vertex, and an edge and a vertex are adjacent if the vertex is one of the endpoints of the edge.

We consider only \textit{simple} graphs — that is, graphs without loops (edges connecting a vertex to itself) and without multiple edges (more than one edge between the same pair of vertices). As we explain below, this apparent limitation is not a real restriction on the generality of our results.

We further restrict our attention to locally finite graphs, namely graphs such that each vertex has a finite number of neighbors. The \textit{degree} of a vertex is the number of edges incident to it. A graph is called \textit{d-regular} if every vertex has degree d.

It will sometimes be convenient to consider $E$ as a set of \textit{oriented edges}, where two directions are assigned to each edge. Given two vertices \( v \) and \( u \), we write \( (v, u) \) for the edge directed from \( v \) to \( u \). For each directed edge \( e \in E \), we denote by \( \check{e} \) the same edge with the opposite orientation.

\begin{notation}\label{E+-}
For a vertex \( v \in V \), we define:
\begin{itemize}
  \item \( E_v \) as the set of all edges incident to \( v \), without orientation;
  \item \( E_v^{\mathrm{+}} = \{ (u, v) \mid u \sim v \} \), the set of edges \textit{incoming} to \( v \);
  \item \( E_v^{\mathrm{-}} = \{ (v, u) \mid u \sim v \} \), the set of edges \textit{outgoing} from \( v \).
\end{itemize}
\end{notation}
A \textit{cycle} is a sequence of vertices \( v_1, \dots, v_n \) such that \( v_i \neq v_j \) for all \( i \neq j \), \( v_i \sim v_{i+1} \) for all \( i \in \{1, \dots, n-1\} \), and \( v_n \sim v_1 \).
A \textit{pure cycle} is a cycle in which all vertices have degree 2, except possibly \( v_1 \), which may have higher degree and serves as the point of connection between the cycle and the rest of the graph.

A \textit{weighted discrete graph} is defined as a tuple \( G = (V, E, a, b) \), where \( (V, E) \) is a discrete graph, \( a : E \to \mathbb{C} \) is a function assigning \textit{edge weights}, and \( b : V \to \mathbb{R} \) is a function assigning a \textit{potential}. We require that \( a_{\check{e}} = \overline{a_e} \) for every \( e \in E \).

This data defines a self-adjoint operator  
\[
A_G : \ell^2(V) \to \ell^2(V)
\]  
called the \textit{Jacobi operator}, which acts as follows:
\begin{equation}
(A\eta)(\upsilon) = b_{\upsilon} \eta(\upsilon) + \sum_{e = (u, \upsilon) \in E^+_{\upsilon}} a_e \, \eta(u) \,.
\label{Discrete_Laplace}
\end{equation}

Given a graph $G=\left(V,E\right)$ and a function $\ell:E\rightarrow\mathbb{R}_{+}=\left\{ x\in\mathbb{R}:x>0\right\} $, we obtain a \textit{metric graph}, namely the corresponding complex for which the length of an edge $e\in E$ is given by $\ell\left(e\right)=\ell(\check{e})=\ell_{e}$. We denote by \( x_e \) the coordinate along the edge \( e \), where we denote the corresponding segment by $I_{e}$. We denote such a metric graph by the letter $\Gamma$ to distinguish it from the combinatorial structure $G$, and write $\Gamma=\left(V,E,\ell\right)$. 

The topological space underlying $\Gamma$ is obtained by $\bigcup_{e\in E}I_{e}/\sim$, where $\sim$ identifies endpoints of these intervals according to the incidence relations determined by the graph $(V,E,\ell)$: each endpoint $0$ or $\ell_e$ is glued to the corresponding vertex $v \in V$.

The metric $d_{\Gamma}$ on $\Gamma$ is defined by
\[
d_{\Gamma}(x,y) := \inf \{ \text{length}(\gamma) \,:\, \gamma \text{ is a path in $\Gamma$ connecting $x$ to $y$} \},
\]
where the length of a path is calculated by summing the Euclidean lengths of its segments in the corresponding edge intervals $I_{e}$. This metric induces the natural topology of $\Gamma$ as a one-dimensional complex.
We will consider only metric graphs where the set of edge lengths $\left\{\ell(e) \right\}_{e\in E}$ is bounded below, namely there is $\delta=\delta_\Gamma>0$ such that $\ell(e)>\delta$ for all $e\in E$. Under this assumption, a \textit{compact metric graph} is a metric graph with a finite number of vertices and edges. 

\medskip
A \textit{quantum graph} $\Gamma$ is a metric graph equipped with
the Hilbert space $\bigoplus_{e\in E}L_{2}\left(e\right)$ and with a
self-adjoint differential operator on it, whose domain is a subspace
of functions in the direct sum of Sobolev spaces $\bigoplus_{e\in E}H^{2}\left(e\right)$,
that satisfy certain conditions on the values and derivatives taken
at the vertices by the edges. In our case, the differential operator with which the graph is equipped is a \textit{Schr\"odinger-type} differential
operator. This is an operator of the form
\begin{equation}
\left(\mathcal{H}_{\Gamma}f\right)\left(x_{e}\right)=-\frac{d^{2}f}{dx_{e}^{2}}\left(x_{e}\right)+W\left(x_{e}\right)f\left(x_{e}\right)\label{Schrodinger_Operator}
\end{equation}
where $W$ is a real-valued function on the graph, called the \textit{potential}.  
We assume that $W$ is continuous on each edge (viewed as a closed interval), but we do not require matching values at the vertices across different edges.
The \textit{standard Schr\"odinger operator} (aka the Laplacian) corresponds to the case $W=0$. 

We denote the domain of \( \mathcal{H}_\Gamma \) by $\operatorname{Dom}\left(\mathcal{H}_{\Gamma}\right)\subseteq\bigoplus_{e\in E}H^{2}\left(e\right)$.
This domain is defined through the `gluing' conditions at the vertices \cite{Introduction_to_Quantum_Graphs}. A certain class of vertex conditions for which \( \mathcal{H}_\Gamma \) is self-adjoint are the \textit{delta-type conditions}. At each vertex \( \upsilon \in V(\Gamma) \), a function \( f = \bigoplus_{e \in E} f_e \in \operatorname{Dom}(\mathcal{H}_\Gamma) \) satisfies the following:
\begin{equation}
\begin{cases}
f \text{ is continuous at } \upsilon, \text{ and}\\
\sum_{e \in E^-_{\upsilon}} \dfrac{d f_e}{dx_e}(\upsilon) = \alpha_{\upsilon} f(\upsilon),
\end{cases}
\label{Delta_Type}
\end{equation}
for some real constant \( \alpha_{\upsilon} \in \mathbb{R} \).  
To define the derivatives in \eqref{Delta_Type}, we fix an orientation on each edge. The derivatives are taken in the direction \textit{outgoing} from the vertex \( \upsilon \), and hence the expression involves the outgoing derivatives along the incident edges. \textit{Kirchhoff's condition} is a special case of the $\delta$-type vertex condition, corresponding to the choice \( \alpha_{\upsilon} = 0 \).  

The class of \textit{extended $\delta$-type} vertex conditions also includes \textit{Dirichlet conditions}, which require:
\begin{equation}
\begin{cases}
f \text{ is continuous at } \upsilon,  \text{ and}\\
f(\upsilon) = 0.
\end{cases}
\label{Dirichlet}
\end{equation}
It is customary to formally denote this condition by \( \alpha_{\upsilon} = \infty \). Unless stated otherwise, when referring to $\delta$-type vertex conditions, we include both standard and extended $\delta$-type conditions.

\begin{remark}
A vertex with Dirichlet boundary conditions effectively separates the graph into several connected components. It is often useful to consider such vertices as leaves (i.e.\ vertices of degree one).
\end{remark}

A \textit{cover} of a quantum graph \( \Gamma = \left(V,E,\ell,\mathcal{H}_{\Gamma}\right) \) is another quantum graph \( \Lambda = (\mathcal{V},\mathcal{E},\ell_{\Lambda},\mathcal{H}_{\Lambda}) \), together with a continuous \textit{covering map} \( \Xi: \Lambda \to \Gamma \), mapping vertices to vertices and edges to edges, which is a local isomorphism preserving edge lengths, vertex gluing data, and potential. Namely the following conditions hold:
\begin{align}
& \forall \upsilon, u \in \mathcal{V},\quad \upsilon \sim u \Rightarrow \Xi(\upsilon) \sim \Xi(u); \label{Covering}  \\
& \forall\, \upsilon\in\mathcal{V},\ \forall\, u\in V, 
\qquad \Xi(\upsilon)\sim u \ \Longrightarrow\ 
\exists\,\tilde{u}\in\Xi^{-1}(u)\ \text{such that}\ \upsilon\sim\tilde{u};
\nonumber\\
& \forall e \in \mathcal{E},\quad \ell_e = \ell_{\Xi(e)}; \nonumber \\
& \forall e \in \mathcal{E},\ \forall x_e \in I_{e},\quad \Xi(x_e) = x_{\Xi(e)}; \nonumber \\
& \forall x \in \Lambda,\quad W(x) = W(\Xi(x)); \nonumber \\
& \forall \upsilon \in \mathcal{V},\quad \alpha_{\upsilon} = \alpha_{\Xi(\upsilon)}. \nonumber \\
 \nonumber
\end{align}
We refer to \( \Lambda \) as a \textit{cover} of \( \Gamma \).  
If \( \Lambda \) is a finite cover, then there exists \( n \in \mathbb{N} \) such that \( |V(\Lambda)| = n \cdot |V(\Gamma)| \), in which case \( \Lambda \) is called an \textit{\(n\)-cover} of \( \Gamma \) \cite{Garza-Vargas}.

\medskip

The \textit{universal cover} of a connected quantum graph \( \Gamma \) is the unique (up to isomorphism) tree \( \mathcal{T} = (\mathcal{V}, \mathcal{E},\ell_{\mathcal{T}},\mathcal{H}_{\mathcal{T}}) \) that covers every other cover of \( \Gamma \).  
If \( \Gamma \) is a tree then $\Gamma=\mathcal{T}$. Otherwise, $\Gamma$ contains at least one cycle, and $\mathcal{T}$ is an infinite tree constructed by unfolding all cycles into infinite, non-repeating paths.

A \textit{periodic quantum tree} is the quantum tree \( \mathcal{T} = (\mathcal{V}, \mathcal{E},\ell_{\mathcal{T}},\mathcal{H}_{\mathcal{T}}) \), obtained as the universal cover of a compact connected quantum graph \( \Gamma = \left(V,E,\ell,\mathcal{H}_{\Gamma}\right) \). 

In the context of discrete graphs $\left(V,E,a,b\right)$, one defines the notion of cover, $\left(\mathcal{V},\mathcal{E},\mathbf{a},\mathbf{b}\right)$ , similarly, where here the covering map preserves vertex potentials and edge weights. That is,
\begin{align*}
&	\forall\upsilon,u\in\mathcal{V},\quad \upsilon\sim u\Rightarrow\Xi(\upsilon)\sim\Xi(u)\ \,\text{and}\ \,\mathbf{a}_{\left(\upsilon,u\right)}=a_{\left(\Xi(\upsilon),\Xi(u)\right)}\\
&	\forall\upsilon\in\mathcal{V},\quad \mathbf{b}_{\upsilon}=b_{\Xi(\upsilon)}\,.
\end{align*}
The notion of universal cover is also defined similarly and the operator obtained on the universal cover of a finite discrete graph is known as a periodic Jacobi matrix on a tree. 
\\

As noted above, without loss of generality we consider only simple graphs. One way to see that this assumption is not restrictive in general, is to note that any loop can be modified by inserting an auxiliary vertex equipped with Kirchhoff conditions, without affecting the function space or the domain \( \operatorname{Dom}(\mathcal{H}_\Gamma) \). A similar procedure can be carried out for eliminating multiple edges. In the particular case of Theorem~\ref{Theorem_regular} below, such a procedure is unacceptable since the insertion of a Kirchhoff vertex of degree 2 destroys regularity. Nevertheless, since we are dealing with an operator on the universal cover of the graph $\Gamma$, we can always take a finite cover of sufficiently high degree where there are no loops or multiple edges (such a cover exists by \cite[Lemma 2.3]{Garza-Vargas}, see also Theorem \ref{ESM_limit} below), and may consider the universal cover as arising from that finite cover. 

Furthermore, with the exception of Theorem~\ref{Theorem_regular}, all results generalize to graphs where the potential is merely piecewise continuous on each edge.
This is achieved, again, by introducing finitely many degree-2 vertices (with Kirchhoff conditions) at the discontinuity points, thereby reducing the problem to the continuous case.
\\

The point spectrum of periodic Jacobi matrices on trees has been studied in \cite{About_Regular_Graphs, Garza-Vargas} and our aim in this paper is to carry out a similar project in the continuum case. 
As remarked above, there are some qualitative differences between the continuum and discrete settings. Mostly these differences are due to the existence of eigenfunctions that do not vanish on an edge but vanish at its two vertices.
For example, \cite{About_Regular_Graphs} shows that a periodic Jacobi matrix on a regular tree has no eigenvalues. In the next example, we show that this result does not hold in the continuum case. 

\begin{example}\label{example_regular}
Let $\Gamma = (V, E, \ell, \mathcal{H}_{\Gamma})$ be the complete quantum graph on 5 vertices, where $\ell_e = 2\pi$ for every $e \in E$. This is a 4-regular quantum graph. We equip $\Gamma$ with the standard Schr\"odinger operator and impose Kirchhoff conditions at all vertices. Let $\mathcal{T} = (\mathcal{V}, \mathcal{E})$ be its universal cover with the standard operator $\mathcal{H}_{\mathcal{T}}$.

$\mathcal{T}$ is an infinite 4-regular tree, with all edges having the same length. Fix a vertex $r \in \mathcal{V}$ and divide the four edges incident to $r$, $(e_1,e_2,e_3,e_4)$, into two pairs: $(e_1,e_2)$, and $(e_3,e_4)$. The subtree $D_+ \subseteq \mathcal{T}$ is obtained by removing $(e_1,e_2)$ and keeping the connected component of $r$ from what remains. Similarly, the subtree $D_-$ is obtained from removing $(e_3,e_4)$ and keeping the connected component of $r$ from what remains. Note that $D_+\cap D_-=\{r\}$.

For each edge $e$ in $\mathcal{T}$ we fix an orientation such that the coordinate $x_e \in [0,2\pi]$ increases away from $r$.
We define an eigenfunction $f$ on $\mathcal{T}$ such that on $D_{+}$, for each edge $e \in \mathcal{E}$ at distance $2\pi n$ from $r$, the function is
\[
f_e(x_e) = \frac{1}{3^n} \sin\left(x_e\right),
\]
and on $D_{-}$, for each edge $e \in \mathcal{E}$ at distance $2\pi n$ from $r$, the function is
\[
f_e(x_e) = -\frac{1}{3^n} \sin\left(x_e\right).
\]
This graph is shown in Figure~\ref{fig:example_regular}.

\begin{center}
        \centering
\begin{tikzpicture}[
    scale=3.5,
    vertex/.style={circle, fill=black, inner sep=1.2pt},
    edge_plus/.style={blue, thick},
    edge_minus/.style={red, thick},
    label_style/.style={font=\normalsize, midway, sloped, yshift=3pt}
]

    \node[vertex, label={[xshift=1pt, yshift=0.2pt]right:{\Large $r$}}] (r) at (0,0) {};

    \foreach \angle/\name in {45/u1, 135/u2} {
        \draw[edge_plus] (r) -- (\angle:0.4) 
            node[vertex] (\name) {}
            node[label_style] {$\sin(x)$};
            
        \foreach \i/\subname in {-1/a, 0/b, 1/c} {
            \pgfmathsetmacro{\newangle}{\angle + \i*38}
            \draw[edge_plus] (\name) -- ($(\name) + (\newangle:0.5)$)
                node[vertex] (\name\subname) {}
                node[label_style, pos=0.6] {$\frac{1}{3}\sin(x)$};
                
            \foreach \j in {-1, 0, 1} {
                \pgfmathsetmacro{\finalangle}{\newangle + \j*22}
                \draw[edge_plus] (\name\subname) -- ($(\name\subname) + (\finalangle:0.75)$)
                    node[vertex] {}
                    node[label_style, pos=0.7] {$\frac{1}{9}\sin(x)$};
            }
        }
    }

    \foreach \angle/\name in {225/d1, 315/d2} {
        \draw[edge_minus] (r) -- (\angle:0.4) 
            node[vertex] (\name) {}
            node[label_style] {$-\sin(x)$};
            
        \foreach \i/\subname in {-1/a, 0/b, 1/c} {
            \pgfmathsetmacro{\newangle}{\angle + \i*38}
            \draw[edge_minus] (\name) -- ($(\name) + (\newangle:0.5)$)
                node[vertex] (\name\subname) {}
                node[label_style, pos=0.6] {$-\frac{1}{3}\sin(x)$};
                
            \foreach \j in {-1, 0, 1} {
                \pgfmathsetmacro{\finalangle}{\newangle + \j*22}
                \draw[edge_minus] (\name\subname) -- ($(\name\subname) + (\finalangle:0.75)$)
                    node[vertex] {}
                    node[label_style, pos=0.6] {$-\frac{1}{9}\sin(x)$};
            }
        }
    }

    \node[blue, font=\Large\bfseries] at (2, 1) {$D_+$};
    \node[red, font=\Large\bfseries] at (2, -1) {$D_-$};

\end{tikzpicture}

    
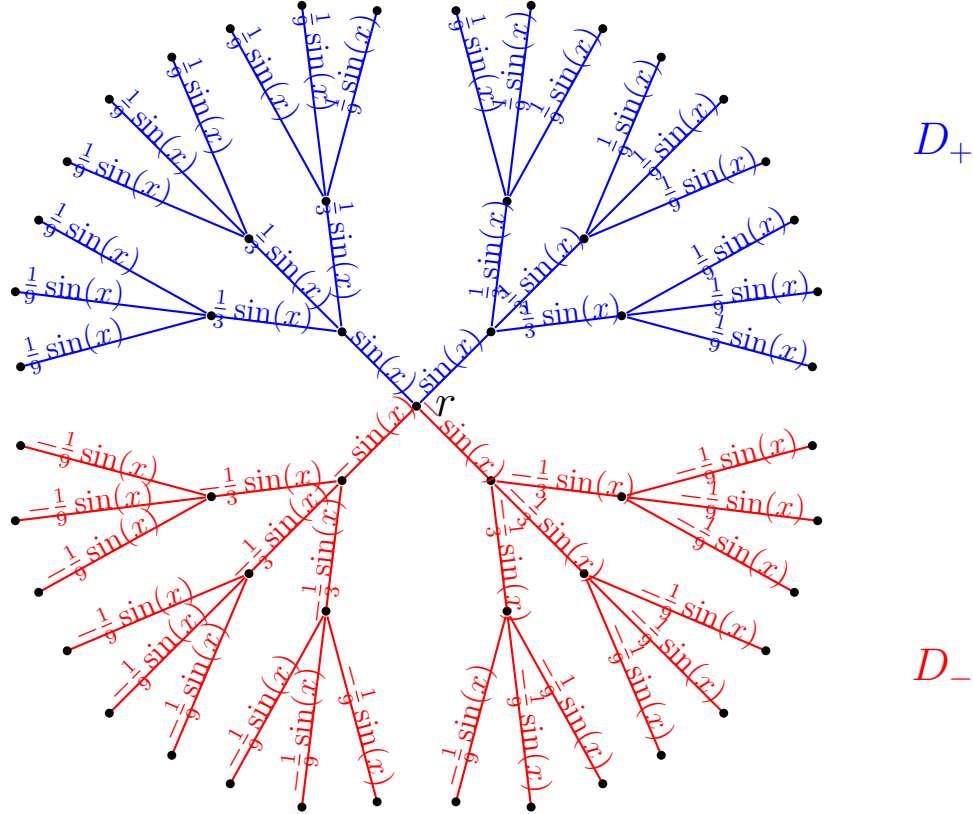
\vspace{1.5em}

    \captionof{figure}{The universal cover from Example~\ref{example_regular}. Depicted is the neighborhood of the vertex $r$. The set $D_{+}$ is shown in blue, and the set $D_{-}$ in red. All sinusoidal functions represent the eigenfunction on each edge, parameterized from the endpoint closer to $r$ to the endpoint farther from $r$.}
    \label{fig:example_regular}
\end{center}
One can see that this function satisfies the vertex conditions $($note that at each vertex $\upsilon \neq r$ there is a single edge $e \in E_\upsilon^-$ directed towards $r)$, and $f \in \bigoplus_{e\in E}H^{2}\left(e\right)$, so $f \in \mathrm{Dom}\left( \mathcal{H}_{\mathcal{T}}\right)$, and it is an eigenfunction of $\mathcal{H}_{\mathcal{T}}$ with eigenvalue 1.

\end{example}

Nevertheless, it is not difficult to characterize all possible eigenfunctions on regular periodic quantum trees. We shall prove the following

\begin{theorem}\label{Theorem_regular}
Let $\Gamma = (V, E, \ell, \mathcal{H}_{\Gamma})$ be a compact $d$-regular simple quantum graph with $\delta$-type vertex conditions, where $\alpha_{\upsilon} < \infty$ for each $\upsilon \in V$. Let $\mathcal{H}_{\Gamma}$ be a Schr\"odinger-type differential operator, and let $\mathcal{T}$ be the universal cover of $\Gamma$ equipped with the corresponding Schr\"odinger operator $\mathcal{H}_{\mathcal{T}}$.

Then for every $\lambda \in \sigma_{\mathfrak{p}} \left(\mathcal{H}_{\mathcal{T}}\right)$ there exists an edge $e \in E$ such that the following differential equation admits a nontrivial solution on $[0, \ell_e]$:
\begin{align}
 & \forall x \in [0, \ell_e],\quad -\frac{d^2 f}{dx^2}(x) + W|_e(x) f(x) = \lambda f(x) \label{Dif_Eq_for_Edge} \\
 & f(0) = f(\ell_e) = 0\,.\nonumber
\end{align}

In the case of the standard Schr\"odinger operator $($i.e., when $W = 0)$, it follows that
\[
\sigma_{\mathfrak{p}} \left(\mathcal{H}_{\mathcal{T}}\right) \subseteq \left\{ \frac{n^2 \pi^2}{\ell_e^2} \mid e \in E,\ n \in \mathbb{N} \right\}\,.
\]
\end{theorem}

\medskip
A central object in the analysis of the point spectrum in the discrete case is the density of states - the average of the spectral measures over a fundamental domain of the action of the symmetry group. In order to formulate the analogous results in the continuum case, we need to devote some space to discuss the definition and properties of the density of states for periodic quantum trees. 

\subsection{THE DENSITY OF STATES}

\begin{definition}\label{The_measures}
\begin{enumerate}[label=\roman*.]
\item \textit{The (normalized) eigenvalue counting measure}, $\nu_{d}$, of a finite discrete graph $G = (V, E, a, b)$
is the measure given by counting the eigenvalues of $A_G$ (with multiplicity)
and dividing by the number of vertices in the graph. That is,
\begin{equation}
\nu_{d} = \frac{1}{\left|V\right|} \sum_{\lambda \in \sigma_{\mathfrak{p}} \left(A_G\right)} \delta_{\lambda}
\end{equation}
where $\delta_{\lambda}$ is the indicator function of $\{ \lambda \}$.

\item \textit{The (normalized) eigenvalue counting measure}, $\nu_{c}$, of a compact quantum graph $\Gamma = (V, E, \ell, \mathcal{H}_{\Gamma})$
is defined similarly to the discrete case, where now the number of eigenvalues (with multiplicity)
is divided by the sum of edge-lengths in the graph. That is,
\begin{equation}
\nu_{c} = \frac{1}{\mathcal{L}_{E}} \sum_{\lambda \in \sigma_{\mathfrak{p}} \left(\mathcal{H}_{\Gamma}\right)} \delta_{\lambda}\,.
\end{equation}
    where $\mathcal{L}_{E}=\sum_{e\in E}\ell_{e}$ is the sum of all edge-lengths in the graph.

\item Let $G = (V, E, a, b)$ be a weighted finite discrete graph and let $\mathcal{T}$ be its universal cover with covering map $\Xi : \mathcal{T} \to \Gamma$.
For every $\upsilon \in V$ choose some $\tilde{\upsilon} \in \Xi^{-1}(\upsilon)$.
Then the \textit{density of states} is defined as
\begin{equation}
\mu_{d} = \frac{1}{\left|V\right|} \sum_{\upsilon \in V} \mu_{\upsilon} \label{D_Den_of_States}
\end{equation}
where $\mu_{\upsilon}$ is the spectral measure of the indicator function $\delta_{\tilde{\upsilon}}$.

\item Let $\Gamma = (V, E , \ell, \mathcal{H}_{\Gamma})$ be a compact quantum graph with a Schr\"odinger-type differential operator and self-adjoint vertex conditions. Let
$X$ be a connected fundamental set of the universal cover, and let
$\{ g_k \}_{k \in \mathbb{N}}$ be an orthonormal basis of $L_2(X) = \bigoplus_{e \in X} L_2(e)$.
Then the \textit{density of states} is
\begin{equation}
\mu_{c} = \frac{1}{\mathcal{L}_{E}} \sum_{k \in \mathbb{N}} \mu_{g_k} \label{Q_Den_of_States}
\end{equation}
where $\mu_{g_k}$ is the spectral measure of $g_k$ and $\mathcal{L}_E=\sum_{e \in E} \ell_e$ is the sum of all edge lengths in $\Gamma$ (as above).
\end{enumerate}
\end{definition}

The proof that $\mu_{c}$ is well defined (namely, is independent of the choice of fundamental domain and orthonormal basis $\{g_k\}_k$) is given in \cite{cite_myself}.
The next theorem elucidates the connection between the eigenvalue counting measure and the density of states. 
\begin{theorem}\label{ESM_limit}
Let $\Gamma$ be a compact quantum graph with a Schr\"odinger-type differential operator, $\mathcal{T}$ its universal cover, and $\mu_{c}$ the density of states of $\mathcal{T}$. Then
\begin{enumerate}[label=\roman*.]
\item There exists a sequence of covers $\Gamma_{n}$ of $\Gamma$ whose girth (that is, the number of edges in its smallest cycle) goes to infinity as $n \to \infty$.
\item For this sequence, the eigenvalue counting measures $\nu_{n}$ of $\Gamma_{n}$ converge vaguely (namely, when integrated against continuous functions of compact support) to $\mu_{c}$.
\end{enumerate}
\end{theorem}

A proof of Theorem \ref{ESM_limit}(i) is given in~\cite[Lemma 2.3]{Garza-Vargas}, where the discrete analogue of Theorem \ref{ESM_limit}(ii) is also proven. The proof of Theorem \ref{ESM_limit}(ii) is given in~\cite{cite_myself}.
In Theorem \ref{DOS_bounds} below we establish bounds on the cumulative distribution function of the density of states, as well as on the measure of finite intervals, in the case of $\delta$-type vertex conditions. This in particular implies that $\mu_{c}$ is a locally finite measure. 

The paper \cite{AISW} studies the limit of the eigenvalue counting measures of finite quantum graphs by defining and studying Benjamini-Schramm limits of quantum graphs (in a much more general context than that of periodic trees). Of course, by the convergence, it follows that our definition of the density of states (in the particular case of periodic graphs) coincides with theirs.

The relevance of the density of states to this paper is through Theorem \ref{Q-Aomoto_first_theorem}(iii), where we compute the value of the density of states at each eigenvalue of the universal cover. The formula we obtain is the continuum analog of a formula that goes back to Aomoto's work \cite{About_Regular_Graphs} in the discrete case (see also \cite{Garza-Vargas}).

\subsection{MAIN RESULTS}

A central object in the analysis of eigenvalues of periodic Jacobi matrices on trees is the Aomoto set corresponding to an eigenvalue, namely the set of vertices in the underlying finite graph whose lifts support an eigenfunction (see \cite{About_Regular_Graphs, Garza-Vargas}). Here we will consider an analogous object, which we call the \emph{Q-Aomoto set}, but we will also need the discrete version. 


\begin{definition}\label{Q-Aomoto_def}

Let $G = (V, E, a, b)$ be a weighted discrete graph and let $\mathcal{T} = \left(\mathcal{V},\mathcal{E},\mathbf{a},\mathbf{b}\right)$ be its universal cover with the covering map $\Xi : \mathcal{T} \to G$. Let $\lambda \in \mathbb{R}$ be an eigenvalue of $\mathcal{T}$. The \emph{Aomoto set} $X_{\lambda}(G)$ is a subset of vertices $\upsilon \in V$ such that there exist $\tilde{\upsilon} \in \Xi^{-1}(\upsilon)$ and a function $\eta \in \ker(\lambda - A_{\mathcal{T}})$ satisfying $\eta(\tilde{\upsilon}) \neq 0$.

\medskip

Let $\Gamma = (V, E, \ell, \mathcal{H}_{\Gamma})$ be a compact quantum graph and let $\mathcal{T} = (\mathcal{V}, \mathcal{E}, \ell_{\mathcal{T}}, \mathcal{H}_{\mathcal{T}})$ be its universal cover with the covering map $\Xi : \mathcal{T} \to \Gamma$. Let $\lambda \in \mathbb{R}$ be an eigenvalue of $\mathcal{T}$. The \emph{Q-Aomoto set} $X_{\lambda}(\Gamma) = V_{\lambda}(\Gamma) \cup E_{\lambda}(\Gamma)$ is a subset of the graph $\Gamma$ defined as follows:
\begin{enumerate}[label=\roman*.]
    \item $V_{\lambda}(\Gamma)$ is the set of all vertices $\upsilon \in V$ such that there exists $\tilde{\upsilon} \in \Xi^{-1}(\upsilon)$ and a function $h \in \ker(\lambda - \mathcal{H}_{\mathcal{T}})$ satisfying $h(\tilde{\upsilon}) \neq 0$.
    
    \item $E_{\lambda}(\Gamma)$ is the set of all edges $e \in E$ such that there exist $\tilde{e} \in \Xi^{-1}(e)$ and a function $h \in \ker(\lambda - \mathcal{H}_{\mathcal{T}})$ for which $h|_{\tilde{e}}$ is not identically zero.
\end{enumerate}
\end{definition}

Note that the Q-Aomoto set $X_\lambda(\Gamma)$ is not necessarily a subgraph of $\Gamma$, but rather a subspace of the metric space $\Gamma$, some of whose points may be vertices. Indeed, in the case of $\delta$-type vertex conditions, continuity implies that if $\upsilon \in V_{\lambda}(\Gamma)$, then every edge incident to $\upsilon$ belongs to $E_{\lambda}(\Gamma)$, i.e.\ $E_{\upsilon} \subseteq E_{\lambda}(\Gamma)$. However there can be an edge (or more) that has an eigenfunction that does not vanish on it but does vanish on at least one of its adjacent vertices. Therefore it is possible for an edge to be in the Q-Aomoto set of the graph, even though some of its adjacent vertices are not.

We define the \emph{boundary} of the Q-Aomoto set by $\partial X_{\lambda}(\Gamma) \cup \partial_{\infty}X_{\lambda}(\Gamma)$, where
\begin{align}
    \partial X_{\lambda}(\Gamma) &= \left\{ \upsilon \in V \setminus V_{\lambda}(\Gamma) \;\middle|\; E_{\upsilon} \cap E_{\lambda}(\Gamma) \neq \emptyset,\ \alpha_{\upsilon} < \infty \right\}, \label{Boundary} \\
    \partial_{\infty}X_{\lambda}(\Gamma) &= \left\{ \upsilon \in V \setminus V_{\lambda}(\Gamma) \;\middle|\; E_{\upsilon} \cap E_{\lambda}(\Gamma) \neq \emptyset,\ \alpha_{\upsilon} = \infty \right\}. \label{Boundary_inf}
\end{align}

We say that two edges $e, e' \in E_{\upsilon} \cap X_{\lambda}(\Gamma)$ (for some $\upsilon \in V$) are \emph{connected in} $X_{\lambda}(\Gamma)$ if $\upsilon \in V_{\lambda}(\Gamma)$.

In general, for a subset $X$ of vertices and edges of a graph, we define a \emph{cycle in $X$} to be a sequence of vertices $\{ \upsilon_1, \ldots, \upsilon_n \} \subseteq V(X)$ and a corresponding sequence of edges 
\[
\left\{ (\upsilon_1, \upsilon_2), (\upsilon_2, \upsilon_3), \ldots, (\upsilon_{n-1}, \upsilon_n), (\upsilon_n, \upsilon_1) \right\} \subseteq E(X)
\]
such that all the listed vertices and edges are contained in $X$. That is, the entire cycle must lie within $X$.

The \emph{index} of $\Gamma$ at $\lambda \in \mathbb{R}$ is defined as
\begin{equation}
    I_{\lambda}(\Gamma) = ccX_{\lambda}(\Gamma) - |\partial X_{\lambda}(\Gamma)|, \label{Index}
\end{equation}
where $ccX_{\lambda}(\Gamma)$ denotes the number of connected components of $X_{\lambda}(\Gamma)$ and $|\partial X_{\lambda}(\Gamma)|$ is the number of vertices in $\partial X_{\lambda}(\Gamma)$ as defined in \eqref{Boundary}.\\


Having defined the basic concepts we need, we proceed to describe properties of eigenvalues of periodic quantum graphs. As mentioned above, these are continuum analogs of the results in \cite{Garza-Vargas}.


\begin{theorem}\label{Q-Aomoto_first_theorem}
Let $\Gamma = (V, E, \ell, \mathcal{H}_{\Gamma})$ be a compact quantum graph with $\delta$-type vertex conditions and a Schr\"odinger-type differential operator $\mathcal{H}_{\Gamma}$. Let $\mathcal{T}$ be the universal cover of $\Gamma$ with covering map $\Xi: \mathcal{T} \to \Gamma$, and let $\lambda \in \sigma_{\mathfrak{p}}(\mathcal{H}_{\mathcal{T}})$.

Then:
\begin{enumerate}[label=\roman*.]
    \item The subgraph of $\Gamma$ induced by the vertices in $V_{\lambda}(\Gamma)$ is acyclic.
    
    \item For each connected component $T_{1}$ of $X_{\lambda}(\Gamma)$ (with our definition  to connection between edges in $X_{\lambda}(\Gamma)$), and for every connected component, $\widetilde{T}_{1}$, of the lift of $T_{1}$ in $\mathcal{T}$
    \[
    \dim \left\{ h|_{\widetilde{T}_{1}} \;\middle|\; h \in \ker(\lambda - \mathcal{H}_{\mathcal{T}}) \right\} = 1.
    \]
    
    \item The density of states satisfies:
    \[
    \mu\left(\{\lambda\}\right) = \frac{I_{\lambda}(\Gamma)}{\mathcal{L}_{E}}.
    \]
\end{enumerate}
\end{theorem}

\begin{remark}
\begin{enumerate}[label=\roman*.]
    \item A graph is called `acyclic' if it contains no cycle, namely if it is a union of trees. In our case, it is not hard to see that the subgraph of $\Gamma$ induced by the vertices in $V_\lambda(\Gamma)$ is acyclic if and only if there exists no cycle in $\Gamma$ whose vertices and edges are all contained in $X_\lambda(\Gamma)$. We shall also refer to this fact by saying that $X_\lambda(\Gamma)$ is \textit{acyclic}.
    \item Since $X_{\lambda}(\Gamma)$ is acyclic, Theorem~\ref{Q-Aomoto_first_theorem}(ii) is equivalent to the statement that $\lambda$ is a simple eigenvalue of the Schr\"odinger operator on $T_{1}$ with Dirichlet vertex conditions imposed on $\partial X_{\lambda}(\Gamma) \cup \partial_{\infty} X_{\lambda}(\Gamma)$.
\end{enumerate}
\end{remark}\label{remark_Q-Aomoto_first_theorem}

Now we examine the connection between the original compact graph and the point spectrum of its universal cover.

\begin{theorem}\label{Q-Aomoto_second_theorem}
Let $\Gamma = (V, E, \ell, \mathcal{H}_{\Gamma})$ be a compact quantum graph with $\delta$-type vertex conditions and a Schr\"odinger-type differential operator $\mathcal{H}_{\Gamma}$.  Let $\mathcal{T}$ be the universal cover of $\Gamma$ with covering map $\Xi : \mathcal{T} \to \Gamma$ and suppose $\lambda \in \sigma_{\mathfrak{p}}(\mathcal{H}_{\mathcal{T}})$. Then:
\begin{enumerate}[label=\roman*.]
    \item $\lambda \in \sigma_{\mathfrak{p}}(\mathcal{H}_{\Gamma})$. Moreover, there exists a subspace of the eigenspace of $\mathcal{H}_{\Gamma}$ and $\lambda$ of dimension at least
    \[
    I_{\lambda}(\Gamma) = \mathcal{L}_{E} \cdot \mu\left(\{\lambda\}\right)
    \]
    whose elements vanish outside the Q-Aomoto set $X_{\lambda}(\Gamma)$.

    \item Let $G$ be an $n$-lift of $\Gamma$ with covering map $\xi : G \to \Gamma$. Then $\lambda \in \sigma_{\mathfrak{p}}(\mathcal{H}_{G})$. Moreover, there exists a subspace of the eigenspace of $\lambda$ in $L_{2}(G)$, of dimension at least
    \[
    n \cdot I_{\lambda}(\Gamma) = n \cdot \mathcal{L}_{E} \cdot \mu\left(\{\lambda\}\right)
    \]
    whose elements vanish outside the lifted Q-Aomoto set $\xi^{-1}(X_{\lambda}(\Gamma))$.
\end{enumerate}
\end{theorem}

\begin{definition}\label{Definition_Eligible_Q-Aomoto}
Let $\Gamma=(V,E,\ell,\mathcal{H}_{\Gamma})$ be a compact quantum graph, let $\lambda \in \mathbb{R}$, and let  $X = E_{X} \cup V_{X} \subseteq \Gamma$, where $V_{X} \subseteq V$ and $E_{X} \subseteq E$. We say that $X$ is an \emph{eligible Q-Aomoto set} of $\lambda$ if the following conditions hold:
\begin{enumerate}[label=\roman*.]
\item
If $\upsilon \in V_{X}$, then $E_{\upsilon} \subseteq E_{X}$.
\item
The vertices in $V_{X}$ induce an acyclic subgraph.
\item
For every connected component $T_{1}$ of $X$, in the sense described for Q-Aomoto sets in Remark~\ref{remark_Q-Aomoto_first_theorem}, the value $\lambda$ is an eigenvalue of the operator $\mathcal{H}_{\Gamma}$ restricted to $T_{1} \cup \partial T_{1} \cup \partial_{\infty} T_{1}$, with Dirichlet boundary conditions imposed on $\partial T_{1}$ and $\partial_{\infty} T_{1}$.
\item
\[
cc(X) - |\partial X| > 0 .
\]
\end{enumerate}
\end{definition}
We denote the family of eligible Q-Aomoto sets of $\lambda$ by $A_\lambda(\Gamma)$. For every $\lambda\in\sigma_{\mathfrak{p}}\left(\mathcal{H}_{\mathcal{T}}\right)$, the Q-Aomoto set $X_{\lambda}(\Gamma)$ is one of the eligible Q-Aomoto sets of $\lambda$. i.e., $X_{\lambda}(\Gamma) \in A_\lambda(\Gamma)$.

\begin{theorem}\label{Q-Aomoto_third_theorem}
Let $\Gamma=\left(V,E, \ell, \mathcal{H}_{\Gamma}\right)$ be a compact quantum graph with $\delta$-type
vertex conditions and with a Schr\"odinger type differential operator
$\mathcal{H}_{\Gamma}$ and let $\mathcal{T}$ be its universal cover.
For every $\lambda\in\mathbb{R}$ and every $X\in A_{\lambda}\left(\Gamma\right)$,
\begin{enumerate}[label=\roman*.]
\item $\lambda\in\sigma_{\mathfrak{p}}\left(\mathcal{H}_{\Gamma}\right)$,
with multiplicity at least $cc\left(X\right)-\left|\partial X\right|$.
\item $\lambda\in\sigma_{\mathfrak{p}}\left(\mathcal{H}_{\mathcal{T}}\right)$ and $\mu\left(\left\{ \lambda\right\} \right)\geq\frac{cc\left(X\right)-\left|\partial X\right|}{\mathcal{L}_{E}}$.
\end{enumerate}
\end{theorem}

\begin{corollary}\label{Q-Aomoto_Corollary}
Let $\Gamma=\left(V,E, \ell, \mathcal{H}_{\Gamma}\right)$ be a compact quantum graph with $\delta$-type
vertex conditions and with a Schr\"odinger type differential operator, and let $\mathcal{T}$ be its universal cover. Then:
\begin{enumerate}[label=\roman*.]
\item For every $\lambda\in\sigma_{\mathfrak{p}}\left(\mathcal{H}_{\mathcal{T}}\right)$,
\[
\mu\left(\left\{ \lambda\right\} \right)=\frac{1}{\mathcal{L}_{E}}\max_{X\in A_{\lambda}\left(\Gamma\right)}\left(cc\left(X\right)-\left|\partial X\right|\right),
\]
where this maximum is acheived by the Q-Aomoto set.
\item $\sigma_{\mathfrak{p}}\left(\mathcal{H}_{\mathcal{T}}\right)$ may
be computed from $\Gamma$ in finite time.
\end{enumerate}
\end{corollary}

From Theorem \ref{Q-Aomoto_first_theorem} and Theorem \ref{Q-Aomoto_third_theorem} we can conclude Corollary \ref{Q-Aomoto_Corollary}(i). And from Theorem \ref{Q-Aomoto_second_theorem} and Theorem \ref{Q-Aomoto_third_theorem} we get Corollary
\ref{Q-Aomoto_Corollary}(ii).

Finally, we show that the existence of non-empty point spectrum for a periodic quantum tree (equipped with the standard Laplacian) is a topologically rare event in the space of possible edge lengths.

\begin{theorem}\label{Berkolaiko_theorem}
Let $\Gamma=\left(V,E, \ell, \mathcal{H}_{\Gamma}\right)$ be a compact quantum graph with $\delta$-type
vertex conditions, where the Dirichlet condition is allowed only at
vertices of degree 1, with at least one cycle and with the standard
Schr\"odinger operator $\mathcal{H}_{\Gamma}=-d^{2}/dx^{2}$, and let
$\mathcal{T}$ be its universal cover with the covering map $\Xi:\mathcal{T}\rightarrow\Gamma$.
Then for a residual set of lengths in $\mathbb{R}_{+}^{\left|E\right|}$,
$\sigma_{\mathfrak{p}}\left(\mathcal{H}_{\mathcal{T}}\right)=\emptyset$.\\
\end{theorem}

The remainder of this paper is structured as follows. In Section 2, we establish some auxiliary results and establish the promised bounds on the density of states of compact intervals.
In Section 3, we prove Theorems~\ref{Theorem_regular}, \ref{Q-Aomoto_first_theorem}, \ref{Q-Aomoto_second_theorem}, \ref{Q-Aomoto_third_theorem}, and \ref{Berkolaiko_theorem}.

\subsection*{Acknolwedgments}
The authors thank Ram Band for many useful suggestions.

\section{AUXILIARY RESULTS}

The main idea behind our analysis is to translate the continuum problem to a discrete one. Namely, given a quantum graph and an eigenvalue $\lambda$, we shall construct a corresponding discrete graph with a Jacobi matrix and a map between $\lambda$-eigenfunctions of the quantum graph and $0$-eigenfunctions of the Jacobi matrix. This idea (applied to compact graphs) goes back at least to Exner in \cite{Exner}, and has been applied in various contexts (see also \cite{AKMN,BGP,Lledo-Post,Post} for some generalizations and applications). 

We generalize the construction in \cite{Exner} in two ways. First, we consider operators with a non-zero potential. Second, we allow for eigenfunctions that vanish at both endpoints of an edge. For this we require a somewhat elaborate construction in the first step of our analysis, which we call the `derived graph' of a quantum graph. This is only necessary for the proof of a central lemma (Lemma \ref{Lemma_From_Derived_Graph}). In Section 3 the derived graph is no longer used, and a simpler correspondence to the discrete case is applied. 

In order to begin, we need to consider a local basis for solutions to the eigenvalue equation.

\begin{notation}\label{f_and_g_Definition}
Let $\Gamma = \left(V, E, \ell, \mathcal{H}_{\Gamma}\right)$ be a quantum graph with $\delta$-type vertex conditions and a Schr\"odinger-type operator $\mathcal{H}_{\Gamma}$ with a potential function $W$. For each edge $e = \left(\upsilon, u\right) \in E$, oriented from $\upsilon$ to $u$, and for each $\lambda \in \mathbb{R}$, consider the following differential equation on $I_{e}=[0,\ell_e]$:
\begin{equation}
-\frac{d^{2}h}{dx^{2}}\left(x_{e}\right)+W|_{e}\left(x_{e}\right)h\left(x_{e}\right)=\lambda h\left(x_{e}\right),\ x_{e}\in I_{e}\,.\label{Diff_Eq_of_e}
\end{equation}
Let us define a real basis for the 2-dimensional solution space $\left\{ f_{e},g_{e}\right\} $,
with the initial conditions:
\begin{align}
f_{e}\left(0\right)=1,\  & f_{e}^{\prime}\left(0\right)=0\label{f_and_g}\\
g_{e}\left(0\right)=0,\  & g_{e}^{\prime}\left(0\right)=1\nonumber 
\end{align}
\end{notation}

Their Wronskian is
\begin{equation}
f_{e}\left(\ell_{e}\right)g_{e}^{\prime}\left(\ell_{e}\right)-f_{e}^{\prime}\left(\ell_{e}\right)g_{e}\left(\ell_{e}\right)=f_{e}\left(0\right)g_{e}^{\prime}\left(0\right)-f_{e}^{\prime}\left(0\right)g_{e}\left(0\right)=1\,.\label{Wronskian}
\end{equation}
Recall Notation~\ref{E+-}. If we know the value of $h \in \ker(\lambda - \mathcal{H}_{\Gamma})$ at a vertex $\upsilon$, and its outgoing derivative in the direction of an edge $e = (\upsilon, u) \in E^{-}_{\upsilon}$, then we can determine the value of $h$ at $u$ and its outgoing derivative at the same edge using the formula:
\begin{equation}
\begin{pmatrix}h\left(u\right)\\
h_{\check{e}}^{\prime}\left(u\right)
\end{pmatrix}=\begin{pmatrix}f_{e}\left(\ell_{e}\right) & g_{e}\left(\ell_{e}\right)\\
-f_{e}^{\prime}\left(\ell_{e}\right) & -g_{e}^{\prime}\left(\ell_{e}\right)
\end{pmatrix}\begin{pmatrix}h\left(\upsilon\right)\\
h_{e}^{\prime}\left(\upsilon\right)
\end{pmatrix},\label{Matrix_f_g}
\end{equation}
where we take $\check{e} = (u, \upsilon)$ to reverse the direction of the derivative, so that it becomes the outgoing derivative from $u$. By the continuity imposed by the $\delta$-type vertex conditions, the limit of $h$ at every $\upsilon \in V$ is the same along every edge $e \in E_{\upsilon}$. Clearly, the application of the matrix from (\ref{Matrix_f_g}) for $e$, followed by the one for $\check{e}$, should give the identity matrix. That is,
\begin{equation}
\begin{pmatrix}f_{\check{e}}\left(\ell_{e}\right) & g_{\check{e}}\left(\ell_{e}\right)\\
-f_{\check{e}}^{\prime}\left(\ell_{e}\right) & -g_{\check{e}}^{\prime}\left(\ell_{e}\right)
\end{pmatrix}\begin{pmatrix}f_{e}\left(\ell_{e}\right) & g_{e}\left(\ell_{e}\right)\\
-f_{e}^{\prime}\left(\ell_{e}\right) & -g_{e}^{\prime}\left(\ell_{e}\right)
\end{pmatrix}=\begin{pmatrix}1 & 0\\
0 & 1
\end{pmatrix}\,.
\end{equation}
We can solve these equations using (\ref{Wronskian}) to obtain the following identities (which will prove useful later):
\begin{align}
 & f_{e}\left(\ell_{e}\right)=g_{\check{e}}^{\prime}\left(\ell_{e}\right)\label{Adjustment_in_the_other_direction}\\
 & g_{e}\left(\ell_{e}\right)=g_{\check{e}}\left(\ell_{e}\right)\nonumber \\
 & f_{e}^{\prime}\left(\ell_{e}\right)=f_{\check{e}}^{\prime}\left(\ell_{e}\right).\nonumber 
\end{align}



We now define the derived graph and the operator $\gamma_{\lambda}$, which establish the connection between the eigenspace of $\lambda$ on the quantum graph and the eigenspace of $0$ on the derived graph.

\begin{definition}\label{Derived_Graph}
(An example is shown in Figure~\ref{fig:derived_graph}) Let $\Gamma=\left(V,E, \ell, \mathcal{H}_{\Gamma}\right)$ be a quantum graph with a Schr\"odinger type differential operator and $\delta$-type vertex conditions. For every $\lambda\in\mathbb{R}$ we define the \textit{derived graph} of $\Gamma$ corresponding to $\lambda$ to be the discrete graph $\Gamma_{disc,\lambda}=\left(V_{disc,\lambda},E_{disc,\lambda},a_{\lambda},b_{\lambda}\right)$ such that:

\begin{align}
V_{disc,\lambda}& =\bigcup_{\underset{\alpha_{\upsilon}<\infty}{\upsilon\in V}}\bigg(\left\{ \upsilon_{p}\right\} \cup\bigcup_{e\in E_{\upsilon}^{-}}\left\{ \upsilon_{e}\right\} \bigg)\cup\bigcup_{\underset{\alpha_{\upsilon}=\infty}{\upsilon\in V}}\bigcup_{e\in E_{\upsilon}^{-}}\left\{ \upsilon_{e}\right\} \label{V_disc}\\
E_{disc,\lambda}& =\bigcup_{\underset{\alpha_{\upsilon},\alpha_{u}<\infty}{e=\left(\upsilon,u\right)\in E}}\left\{ \left(\upsilon_{p},u_{p}\right),\left(\upsilon_{p},u_{\check{e}}\right),\left(\upsilon_{p},\upsilon_{e}\right),\left(\upsilon_{e},u_{\check{e}}\right)\right\} \cup\label{E_disc}\\
 & \bigcup_{\underset{\alpha_{\upsilon}<\infty,\alpha_{u}=\infty}{e=\left(\upsilon,u\right)\in E}}\left\{ \left(\upsilon_{p},u_{\check{e}}\right),\left(\upsilon_{p},\upsilon_{e}\right),\left(\upsilon_{e},u_{\check{e}}\right)\right\} \cup\bigcup_{\underset{\alpha_{\upsilon},\alpha_{u}=\infty}{e=\left(\upsilon,u\right)\in E}}\left\{ \left(\upsilon_{e},u_{\check{e}}\right)\right\}  \nonumber 
\end{align}

Thus, for each vertex $\upsilon \in \Gamma$ with non-Dirichlet boundary conditions, the graph $\Gamma_{disc, \lambda}$ has a copy of $\upsilon$ corresponding to the original vertex, as well as copies of $\upsilon$ for each outgoing edge. In the case of Dirichlet boundary conditions, the only copies of $\upsilon$ are the ones corresponding to the emanating edges. We call $\upsilon_{p}$ the \textit{principal vertex} of $\upsilon$
and call $\left\{ \upsilon_{e}\right\} _{e\in E_{\upsilon}^{-}}$ the \textit{shadow
vertices} of $\upsilon$. 


As for the edges, these echo the structure of the graph with connections between principal vertices (when they exist), principal vertices and shadow vertices, and among shadow vertices, all replicating those in $\Gamma$. In addition, we connect between every principal edge and its shadow edges.

For every $\upsilon\in V$ and every $e\in E_{\upsilon}$ the potential satisfies
\begin{equation}
b_{\upsilon_{i}}=\begin{cases}
2\alpha_{\upsilon} & i=p\\
0 & i\neq p
\end{cases}\,,\label{Potential_in_v}
\end{equation}
and for every $e=\left(u,\upsilon\right)\in E$,
\begin{equation} \label{Weight_in_v_u}
\begin{split}
a_{\left(\upsilon_{p},\upsilon_{\check{e}}\right)}&=a_{\left(\upsilon_{\check{e}},\upsilon_{p}\right)}=-1\\
  a_{\left(u_{p},\upsilon_{p}\right)}&=f_{e}^{\prime}\left(\ell_{e}\right) \\
  a_{\left(u_{e},\upsilon_{\check{e}}\right)}&=g_{e}\left(\ell_{e}\right) \\
  a_{\left(u_{p},\upsilon_{\check{e}}\right)}&=f_{e}\left(\ell_{e}\right) \\ 
  a_{\left(u_{e},\upsilon_{p}\right)}&=g_{e}^{\prime}\left(\ell_{e}\right).
\end{split}
\end{equation}

This gives us the Jacobi matrix 
\begin{equation}
\left(A_{\Gamma}\eta\right)\left(\upsilon_{j}\right)=\begin{cases}
2\alpha_{\upsilon}\eta\left(\upsilon_{p}\right)+\underset{e=\left(u,\upsilon\right)\in E_{\upsilon}^{+}}{\sum}\left(f_{e}^{\prime}\left(\ell_{e}\right)\eta\left(u_{p}\right)+g_{e}^{\prime}\left(\ell_{e}\right)\eta\left(u_{e}\right)-\eta\left(\upsilon_{\check{e}}\right)\right) & j=p\\
f_{e}\left(\ell_{e}\right)\eta\left(u_{p}\right)+g_{e}\left(\ell_{e}\right)\eta\left(u_{e}\right)-\eta\left(\upsilon_{p}\right) & j=\check{e}:e=\left(u,\upsilon\right)\in E_{\upsilon}^{+}
\end{cases}\label{A_Gamma}
\end{equation} 
This formula holds for $\alpha_{\upsilon}, \alpha_{u} < \infty$. In the case $\alpha_{\upsilon} = \infty$ ($\alpha_{u} = \infty$), replace $\eta(\upsilon_{p})$ $(\eta(u_{p}))$ by 0.
\end{definition}

Notice that the weights are real, and by~\eqref{Adjustment_in_the_other_direction}, they do not depend on the direction of the edge.

\begin{center}
    \begin{minipage}{0.35\textwidth}
        \centering
        \begin{tikzpicture}[scale=0.6, every node/.style={circle, draw}]
            \node[circle, draw=black, fill=red!60, minimum size=1cm] (A) at (0,2) {\textbf{A}};
            \node[circle, draw=black, fill=blue!60, minimum size=1cm] (B) at (2,0) {\textbf{B}};
            \node[circle, draw=black, fill=brown!60, minimum size=1cm] (C) at (-2,0) {\textbf{C}};
            \node[circle, draw=black, fill=yellow!60, minimum size=1cm] (D) at (1,-2.3) {\textbf{D}};

            \draw[line width=2pt] (A) -- (B);
            \draw[line width=2pt] (A) -- (C);
            \draw[line width=2pt] (A) -- (D);
            \draw[line width=2pt] (C) -- (B);
        \end{tikzpicture}
        \captionof*{figure}{Graph 1}
    \end{minipage}
    \hfill
    \begin{minipage}{0.6\textwidth}
        \centering
        \begin{tikzpicture}[scale=0.4, every node/.style={circle, draw}] 
            \node[circle, draw=black, line width=1mm, fill=red!60, minimum size=1cm] (A1) at (0,9) {$A_{p}$};
            \node[circle, draw=black, line width=1mm, fill=blue!60, minimum size=1cm] (B1) at (6,5.5) {$B_{p}$};
            \node[circle, draw=black, line width=1mm, fill=brown!60, minimum size=1cm] (C1) at (-6,5.5) {$C_{p}$};
            \node[circle, draw=black, line width=1mm, fill=yellow!60, minimum size=1cm] (D1) at (3.5,-4) {$D_{p}$};

            \node[circle, scale=0.65, draw, fill=red!60, minimum size=1cm] (A2) at (0,9-2*2.5) {$A_{(A,B)}$};
            \node[circle, scale=0.65, draw, fill=red!60, minimum size=1cm] (A3) at (0,9-3*2.5) {$A_{(A,C)}$};
            \node[circle, scale=0.65, draw, fill=red!60, minimum size=1cm] (A4) at (0,9-4*2.5) {$A_{(A,D)}$};
            \node[circle, scale=0.65, draw, fill=blue!60, minimum size=1cm] (B2) at (6,5.5-2*2.5) {$B_{(B,A)}$};
            \node[circle, scale=0.65, draw, fill=blue!60, minimum size=1cm] (B3) at (6,5.5-3*2.5) {$B_{(B,C)}$};
            \node[circle, scale=0.65, draw, fill=brown!60, minimum size=1cm] (C2) at (-6,5.5-2*2.5) {$C_{(C,A)}$};
            \node[circle, scale=0.65, draw, fill=brown!60, minimum size=1cm] (C3) at (-6,5.5-3*2.5) {$C_{(C,B)}$};
            \node[circle, scale=0.65, draw, fill=yellow!60, minimum size=1cm] (D2) at (3.5,-4-2*2.5) {$D_{(D,A)}$};

            \draw[-] (A1) -- (A2); \draw[-, bend right=30] (A1) to (A3); \draw[-, bend right=45] (A1) to (A4);
            \draw[-] (B1) -- (B2); \draw[-, bend left=30] (B1) to (B3);
            \draw[-] (C1) -- (C2); \draw[-, bend right=30] (C1) to (C3);
            \draw[-] (D1) -- (D2);
            \draw[-] (A1) -- (B1); \draw[-] (A1) -- (C1); \draw[-] (A1) -- (D1); \draw[-] (C1) -- (B1);
            \draw[-] (A2) -- (B1); \draw[-] (A2) -- (B2); \draw[-] (A3) -- (C1); \draw[-] (A3) -- (C2);
            \draw[-] (A4) -- (D1); \draw[-] (A4) -- (D2); \draw[-] (B2) -- (A1); \draw[-] (B2) -- (A2);
            \draw[-] (B3) -- (C1); \draw[-] (B3) -- (C3); \draw[-] (C2) -- (A1); \draw[-] (C2) -- (A3);
            \draw[-] (C3) -- (B1); \draw[-] (C3) -- (B3); \draw[-] (D2) -- (A1); \draw[-] (D2) -- (A4);
        \end{tikzpicture}
        \captionof*{figure}{Graph 2}
    \end{minipage}

    \vspace{1.5em}

    
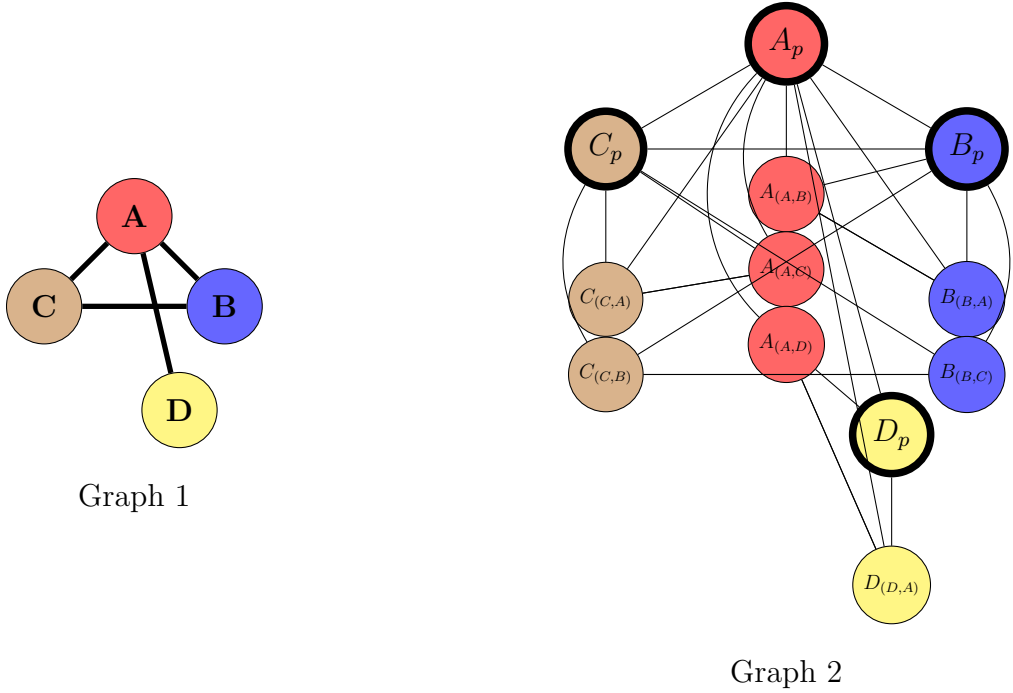
\captionof{figure}{Graph 1 is a quantum graph $\Gamma = (V, E, \ell, \mathcal{H}_{\Gamma})$ where $\alpha_{\upsilon} < \infty$ for each $\upsilon \in V$, and Graph 2 is its derived graph. For every $\upsilon \in V$, we separated between the principal vertex and the shadow vertices. Each vertex $\upsilon$ has $\deg(\upsilon)$ shadow vertices, each corresponding to a different edge incident to $\upsilon$ in $\Gamma$. Every shadow vertex is connected to the principal vertex of the neighboring vertex to which it was connected in $\Gamma$, and to the shadow vertex corresponding to the same edge with the opposite orientation. The principal vertices are also connected to one another if they were connected in $\Gamma$, and each principal vertex is connected to all of its corresponding shadow vertices.}
    \label{fig:derived_graph}
\end{center}

Now we would like to define a one-to-one linear operator $\gamma_{\lambda}:\ker\left(\lambda - \mathcal{H}\right) \hookrightarrow \ker\left(A_{\Gamma}\right)$, where for each \( h \in \ker\left(\lambda - \mathcal{H}\right) \) and each \( \upsilon \in V \),
\[
\gamma_{\lambda}(h)(\upsilon_{p}) = h(\upsilon),
\]
and for each \( e \in E_{\upsilon}^{-} \),
\[
\gamma_{\lambda}(h)(\upsilon_{e}) = h_{e}'(\upsilon),
\]
where \( h_{e}'(\upsilon) \) is the outgoing derivative of \( h \) at \( \upsilon \) in the direction of \( e \). Before proceeding, we must verify that this indeed defines an injective linear map into \( \ker(A) \), and that it preserves the finiteness of the norm.

\begin{lemma}\label{Gamma_is_Embedding}
Let $\Gamma = (V, E, \ell, \mathcal{H}_{\Gamma})$ be a quantum graph equipped with a Schr\"odinger operator and $\delta$-type vertex conditions. Assume that:
\begin{itemize}
  \item The edge lengths $\{ \ell_e \}_{e \in E}$ take values in a finite set,
  \item The potential functions $\{ W|_e \}_{e \in E}$ take values in a finite set.
\end{itemize}
Then for every eigenvalue $\lambda \in \mathbb{R}$, the map $\gamma_{\lambda}$ is an injective linear map from $\ker\left( \lambda - \mathcal{H} \right)$ to $\ker(A_{\Gamma})$.
\end{lemma}

\begin{proof}
First, we want to prove that the operator preserves the finiteness of the norm. That is, we want to prove that
\[
\sum_{\underset{\alpha_{\upsilon}<\infty}{\upsilon\in V}}\left(\left|\gamma_{\lambda}\left(h\right)\left(\upsilon_{p}\right)\right|^{2}+\sum_{e\in E_{\upsilon}^{-}}\left|\gamma_{\lambda}\left(h\right)\left(\upsilon_{e}\right)\right|^{2}\right)+\sum_{\underset{\alpha_{\upsilon}=\infty}{\upsilon\in V}}\sum_{e\in E_{\upsilon}^{-}}\left|\gamma_{\lambda}\left(h\right)\left(\upsilon_{e}\right)\right|^{2}<\infty\,.
\]
We assume that $h \in \ker(\lambda - \mathcal{H})$ satisfies $\| h \|_{L^{2}(\Gamma)} < \infty$ and $\| h' \|_{L^{2}(\Gamma)} < \infty$. According to~\eqref{f_and_g}, the function $h_e$ on each edge $e$ admits the representation
\[
h_{e}(x) = h_{e}(0) f_{e}(x) + h_{e}'(0) g_{e}(x), \qquad x \in [0, \ell_{e}]\,.
\]
Since the functions $f_e$ and $g_e$ are continuous, there exists a constant $\ell_{1} > 0$, which may be chosen such that $\ell_{1} < \frac{1}{2} \min_{e \in E} \ell_e$, and such that for all $e \in E$ and all $x \in [0, \ell_{1}]$, the following uniform estimate holds:
\[
\max \left\{ \left| f_{e}(x) - 1 \right|, \left| f_{e}'(x) \right|, \left| g_{e}(x) \right|, \left| g_{e}'(x) - 1 \right| \right\} < \frac{1}{10}\,.
\]
Using
\begin{equation*}
2\textrm{Re}\left(a\overline{b}\right)\leq\left|a\right|^{2}+\left|b\right|^{2},\ \  \forall a,b\in\mathbb{C}
\end{equation*}
we get
\begin{align*}
 & \int_{0}^{\ell_{1}}\left|h_{e}\left(x\right)\right|^{2}dx=\int_{0}^{\ell_{1}}\left|h_{e}\left(0\right)\right|^{2}\left(f_{e}\left(x\right)\right)^{2}+2\Re\left(h_{e}\left(0\right)\overline{h_{e}^{\prime}\left(0\right)}\right)f_{e}\left(x\right)g_{e}\left(x\right)+\left|h_{e}^{\prime}\left(0\right)\right|^{2}\left(g_{e}\left(x\right)\right)^{2}dx\geq\\
 & \geq\int_{0}^{\ell_{1}}\left|h_{e}\left(0\right)\right|^{2}\left(\frac{9}{10}\right)^{2}-\left(\left|h_{e}\left(0\right)\right|^{2}+\left|h_{e}^{\prime}\left(0\right)\right|^{2}\right)\frac{11}{10}\frac{1}{10}dx=\frac{\ell_{1}}{100}\left(70\left|h_{e}\left(0\right)\right|^{2}-11\left|h_{e}^{\prime}\left(0\right)\right|^{2}\right)\,,
\end{align*}
and in the same way
\[
\int_{0}^{\ell_{1}}\left|h_{e}^{\prime}\left(x\right)\right|^{2}dx\geq\frac{\ell_{1}}{100}\left(70\left|h_{e}^{\prime}\left(0\right)\right|^{2}-11\left|h_{e}\left(0\right)\right|^{2}\right)\,.
\]
So
\[
\int_{0}^{\ell_{1}}\left|h_{e}\left(x\right)\right|^{2}dx+\int_{0}^{\ell_{1}}\left|h_{e}^{\prime}\left(x\right)\right|^{2}dx\geq\frac{59\ell_{1}}{100}\left(\left|h_{e}\left(0\right)\right|^{2}+\left|h_{e}^{\prime}\left(0\right)\right|^{2}\right)\,.
\]
Summing the contributions over all relevant edge neighborhoods, we obtain the estimate
\[
\frac{59\ell_{1}}{100}\left\Vert \gamma_{\lambda}\left(h\right)\right\Vert _{\ell_{2}\left(\Gamma_{disc,\lambda}\right)}^{2}\leq\left\Vert h\right\Vert _{L_{2}\left(\Gamma\right)}^{2}+\left\Vert h^{\prime}\right\Vert _{L_{2}\left(\Gamma\right)}^{2}<\infty\,.
\]

Now it remains to prove that $\text{Im}(\gamma_{\lambda}) \subseteq \ker\left(A_\Gamma \right)$. From~\eqref{Matrix_f_g}, every eigenfunction $h \in \ker\left(\lambda - \mathcal{H}\right)$ and every pair $\upsilon \sim u$, with $e = (u, \upsilon)$, satisfy
\[
\gamma_{\lambda}(h)(\upsilon_p) = f_{e}(\ell_e) \cdot \gamma_{\lambda}(h)(u_p) + g_{e}(\ell_e) \cdot \gamma_{\lambda}(h)(u_e)\,.
\]
Therefore,
\[
\left(A_{\Gamma} \gamma_{\lambda}(h)\right)(\upsilon_{\check{e}}) = 0
\quad \text{for every } \upsilon \in V \text{ and every } e \in E_{\upsilon}^{+}\,.
\]

Moreover, from the vertex conditions~\eqref{Delta_Type} we have:
\[
\alpha_{\upsilon} \cdot \gamma_{\lambda}(h)(\upsilon_p)
= \sum_{e = (u,\upsilon) \in E_{\upsilon}^{+}} \gamma_{\lambda}(h)(\upsilon_{\check{e}})
= -\sum_{e = (u,\upsilon) \in E_{\upsilon}^{+}} \left[
f_{e}'(\ell_e) \cdot \gamma_{\lambda}(h)(u_p)
+ g_{e}'(\ell_e) \cdot \gamma_{\lambda}(h)(u_e)
\right]\,.
\]
Hence
\[
\left(A_{\Gamma} \gamma_{\lambda}(h)\right)(\upsilon_p) = 0
\quad \text{for every } \upsilon \in V\,.
\]
\end{proof}

The following Lemma plays a central role in our analysis through the proof of Theorem \ref{Q-Aomoto_first_theorem}. Its proof utilizes the map $\gamma_{\lambda}$.

\begin{lemma}\label{Lemma_From_Derived_Graph}
Let $\mathcal{T} = (V, E, \ell, \mathcal{H})$ be a connected quantum tree equipped with a Schr\"odinger-type operator $\mathcal{H}$ and $\delta$-type vertex conditions on all vertices, such that $\alpha_{\upsilon} < \infty$ for every vertex $\upsilon$ with $\deg(\upsilon) > 1$. Denote by $\partial_{\infty} \mathcal{T}$ the set of vertices $\upsilon$ for which $\alpha_{\upsilon} = \infty$ $($so $\partial_{\infty} \mathcal{T} \subseteq \left\{ \upsilon \in V \mid \deg(\upsilon) = 1 \right\})$.

Suppose $\lambda \in \sigma_{\mathfrak{p}}(\mathcal{H})$ and that there exists an eigenfunction $h \in \ker(\lambda - \mathcal{H})$ such that $h(\upsilon) \neq 0$ for all $\upsilon \in V \setminus \partial_{\infty} \mathcal{T}$. Then the following statements hold:
\begin{enumerate}[label=\textnormal{(\roman*)}]
    \item The set $\left\{ f|_{V} \mid f \in \ker(\lambda - \mathcal{H}) \right\}$ is a one-dimensional subspace of $\ell^{2}(V)$.
    \item If $\mathcal{T}$ is finite, then $\dim\left(\ker(\lambda - \mathcal{H})\right) = 1$.
\end{enumerate}
\end{lemma}

\begin{proof}[Proof]
Let us take the derived graph $\mathcal{T}_{disc,\lambda}=\left(V_{disc,\lambda},E_{disc,\lambda},a_{\lambda},b_{\lambda}\right)$
with Jacobi operator $A_{\mathcal{T}}$, choose some root $r\in V\backslash\partial_\infty\mathcal{T}$ 
and denote $\eta=\gamma_{\lambda}\left(h\right)$.

Let $\tilde{h}\in\ker\left(\lambda-\mathcal{H}\right)$, denote $\zeta=\gamma_{\lambda}\left(\tilde{h}\right)$
and consider $c=\frac{\zeta(r_{p})}{\eta(r_{p})}$. We will prove
by induction on the distance from $r$, that $\zeta\left(\upsilon_{p}\right)=c\eta\left(\upsilon_{p}\right)$
for each $\upsilon\in V$, showing (i). We shall then show that if in addition $\mathcal{T}$ is finite, then
$\zeta\left(u\right)=c\eta\left(u\right)$ for each $u\in V_{disc,\lambda}$. By the injectivity of the mapping $\gamma_\lambda$, this will imply (ii).

Assume that $\zeta\left(\upsilon_{p}\right)=c\eta\left(\upsilon_{p}\right)$
and denote $e=\left(u,\upsilon\right)$ for $u\sim\upsilon$, where $\upsilon$ is the edge closer to $r$ in the graph. If $g_{e}\left(\ell_{e}\right)=0$,
then by the assumption and (\ref{Matrix_f_g}),
\[
cf_{e}\left(\ell_{e}\right)\eta\left(u_{p}\right)=c\eta\left(\upsilon_{p}\right)=\zeta\left(\upsilon_{p}\right)=f_{e}\left(\ell_{e}\right)\zeta\left(u_{p}\right)\,.
\]
In this case, by (\ref{Wronskian}) we know that $f_{e}\left(\ell_{e}\right)\neq0$,
so $\zeta\left(u_{p}\right)=c\eta\left(u_{p}\right)$.

Assume that $g_{e}\left(\ell_{e}\right)\neq0$. Consider the subtree $\Gamma$ of $\mathcal{T}$ obtained by deleting all edges emanating from $\upsilon$ except for $e$ (connecting it to $u$) and keeping the connected component that contains $\upsilon$ and $u$. Note that in the derived graph of $\Gamma$ $\upsilon$ has only one shadow vertex - $\upsilon_{\check{e}}$. Let $T_{\upsilon,e}$ be the graph obtained from this derived graph by removing $\upsilon_{p}$ (and the edges emanating from it), but keeping $\upsilon_{\check{e}}$. This is a subgraph of $\mathcal{T}_{disc,\lambda}$.

Then, because $A_{\mathcal{T}}$ is self-adjoint we have
\begin{equation}
\left\langle \zeta,A_{\mathcal{T}}\left(\eta|_{T_{\upsilon,e}}\right)\right\rangle =\left\langle A_{\mathcal{T}}\zeta,\eta|_{T_{\upsilon,e}}\right\rangle =0\,.\label{Inner_Product}
\end{equation}
Since $A_{\mathcal{T}}\eta = 0$ by Lemma \ref{Gamma_is_Embedding}, we obtain the formula
\begin{align}
& A_{\mathcal{T}}\left(\eta|_{T_{\upsilon,e}}\right)=\left[-\eta\left(\upsilon_{\check{e}}\right)+f_{e}^{\prime}\left(\ell_{e}\right)\eta\left(u_{p}\right)+g_{e}^{\prime}\left(\ell_{e}\right)\eta\left(u_{e}\right)\right]\delta_{\upsilon_{p}}+\label{A_Red_Func}\\
& +\eta\left(\upsilon_{p}\right)\delta_{\upsilon_{\check{e}}}-f_{e}^{\prime}\left(\ell_{e}\right)\eta\left(\upsilon_{p}\right)\delta_{u_{p}}-g_{e}^{\prime}\left(\ell_{e}\right)\eta\left(\upsilon_{p}\right)\delta_{u_{e}}\,, \nonumber 
\end{align}
where we used (\ref{Adjustment_in_the_other_direction}). Plug (\ref{A_Red_Func}) into (\ref{Inner_Product}), and because $\eta\left(\upsilon_{p}\right)\neq0$
we deduce that
\begin{equation}
\left[\zeta\left(\upsilon_{\check{e}}\right)-c\eta\left(\upsilon_{\check{e}}\right)\right]-f_{e}^{\prime}\left(\ell_{e}\right)\left[\zeta\left(u_{p}\right)-c\eta\left(u_{p}\right)\right]+g_{e}^{\prime}\left(\ell_{e}\right)\left[c\eta\left(u_{e}\right)-\zeta\left(u_{e}\right)\right]=0\,.\label{mid_calculation}
\end{equation}
Now, use the equations $\left[A_{\mathcal{T}}\zeta\right]\left(\upsilon_{\check{e}}\right) = \left[A_{\mathcal{T}}\zeta\right]\left(u_{e}\right) = 0$ and (\ref{Adjustment_in_the_other_direction}) to write $$\zeta\left(\upsilon_{\check{e}}\right)=\frac{1}{g_e(\ell_e)}\left(\zeta\left(u_p \right)-g_e'(\ell_e)\zeta\left(\upsilon_p \right) \right),$$ 
and 
$$\zeta\left(u_{e}\right)=\frac{1}{g_e(\ell_e)}\left(\zeta\left(\upsilon_p\right)-f_e(\ell_e)\zeta\left(u_p \right) \right).$$ 
Substitute these into (\ref{mid_calculation}), and apply (\ref{Wronskian}) to obtain:
\begin{equation}
\zeta\left(u_{p}\right) = \frac{c g_{e}(\ell_{e})}{2} \left[\frac{2 g'_{e}(\ell_{e})}{g_{e}(\ell_{e})} \eta(\upsilon_{p}) + \eta(\upsilon_{\check{e}}) - f_{e}'(\ell_{e}) \eta(u_{p}) - g'_{e}(\ell_{e}) \eta(u_{e}) \right].\label{mid_calculation_2}
\end{equation}
Next, use $\left[A_{\mathcal{T}}\eta\right]\left(u_{e}\right) = 0$ and $\left[A_{\mathcal{T}}\eta\right]\left(\upsilon_{\check{e}}\right) = 0$ to derive similar expressions for $\eta\left(\upsilon_{\check{e}}\right)$ and $\eta\left(\upsilon_{p}\right)$. Substitute these into (\ref{mid_calculation_2}) to conclude that $\zeta\left(u_{p}\right) = c\eta\left(u_{p}\right)$, thus completing the proof of (i).

For (ii), if $g_{e}\left(\ell_{e}\right)\neq0$, insert $\zeta\left(u_{p}\right) = c\eta\left(u_{p}\right)$ into the formulas we have found for $\zeta(\upsilon_{\check{e}})$ and $\eta(\upsilon_{\check{e}})$ to obtain $\zeta(\upsilon_{\check{e}}) = c \eta(\upsilon_{\check{e}})$. Finally, using (\ref{A_Gamma}) for $\left[A_{\mathcal{T}}\eta\right]\left(\upsilon_{\check{e}}\right)$ and the formula for $\zeta(u_{e})$, we deduce that $\zeta(u_{e}) = c \eta(u_{e})$.

We have thus completed the proof for every edge $e \in E$ for which $g_{e}(\ell_{e}) \neq 0$. Note that, by (\ref{Adjustment_in_the_other_direction}), this property does not depend on the direction of the edge.

If $\mathcal{T}$ is finite, consider the subgraph $G$ spanned by the edges $e \in E$ such that $g_{e}(\ell_{e}) = 0$ and let $T$ be a connected component of $G$. Suppose one of the leaves (i.e.\ vertices of degree one), $w$, of this subgraph has Dirichlet conditions (so, in particular, it is also a leaf of $\mathcal{T}$) and suppose that $w$ is connected to the rest of the subgraph $G$ through the vertex $u$. Denote $e = \left(w,u\right)$ (where $\deg\left(w\right) = 1$ and $u$ is the only vertex connected to $w$). From (\ref{A_Gamma}), because  $g_{e}(\ell_{e}) = 0$, $\eta(w_{p}) \neq 0$ if and only if $\eta(u_{p}) \neq 0$. Therefore $\eta(u_{p}) = 0$, and if $\alpha_{u} < \infty$ we have a contradiction to our assumption about $\eta$. If $\alpha_{u} = \infty$, then there are only two vertices and one edge in $\mathcal{T}$, and the claim is automatically true since $\mathcal{T}$ is an interval in this case.

Thus, we may assume that $\alpha_{u} < \infty$ for every vertex $u$ of $T$. Now, again for a leaf, $w$, of $T$,  let $e = \left(w,u\right)$ be the unique edge connecting $w$ to the rest of $T$. Note that for any edge $e'\neq e$ emanating from $w$ in $\mathcal{T}$ (if such an edge exists) we have that $\zeta\left(w_{e'}\right)=c\eta\left(w_{e'}\right)$. In addition $\zeta\left(w_p\right)=c\eta\left(w_p\right)$.
It now follows from the vertex condition given by $\alpha_w<\infty$ that  $\zeta\left(w_{e}\right)=c\eta\left(w_{e}\right)$. By following $h$ and $\tilde{h}$ along the edge $\check{e}$, we get that $\tilde{h}'(u)=ch'(u)$, namely $\zeta\left(u_{\check{e}}\right)=c\eta\left(u_{\check{e}}\right)$. We can repeat this procedure starting from all leaves and then remove the leaves from $T$. This will give us a new finite tree, so it has leaves. Repeating this procedure will eventually exhaust the finite graph. Thus we are done.
\end{proof}

The following theorem, which is of independent interest, presents bounds on the DOS. We shall use its corollary below.

\begin{theorem}\label{DOS_bounds}
Let $\Gamma = \left(V,E, \ell, \mathcal{H}_{\Gamma}\right)$ be a compact quantum graph with a Schr\"odinger-type differential operator and $\delta$-type vertex conditions, and let $\mathcal{T} = (\mathcal{V}, \mathcal{E}, \ell_{\mathcal{T}}, \mathcal{H}_{\mathcal{T}})$ be its universal cover. Let $\mu$ be the density of states associated with $\mathcal{T}$ and denote
\[
W_{\max} = \max_{x \in \Gamma} W(x), \quad 
W_{e,\max} = \max_{x \in I_{e}} W(x), \quad 
W_{e,\min} = \min_{x \in I_{e}} W(x), \quad 
V^{\circ} = \left\{ \upsilon \in V \mid \alpha_{\upsilon} < \infty \right\}\,,
\]
where we refer to $V^{\circ}$ as the \textit{interior} of $\Gamma$. Then:

\begin{enumerate}[label=\roman*.]
\item For every $x > W_{\max}$,
\begin{align}
&  \max\left\{ \frac{\sum_{e\in E}\ell_{e}\sqrt{x-W_{e,\max}}}{\pi\mathcal{L}_{E}}-\frac{\left|E\right|}{\mathcal{L}_{E}},0\right\}\leq\mu\bigl(\left(-\infty,x\right)\bigl)\label{Bound_commulative}\\
 & 
 \leq\mu\bigl(\left(-\infty,x\right]\bigl)\leq\frac{\sum_{e\in E}\ell_{e}\sqrt{x-W_{e,\min}}}{\pi\mathcal{L}_{E}}+\frac{\left|V^{\circ}\right|}{\mathcal{L}_{E}}\,.\nonumber
\end{align}
In particular, the following Weyl-type asymptotics hold
\begin{equation}\label{Asymptote_commulative}
\mu\bigl((-\infty, x]\bigl) = \frac{\sqrt{x}}{\pi}+O\left(1\right)\qquad x \to \infty.
\end{equation}

\item For every $W_{\max} < a \leq b < \infty$,
\begin{align}
 & \max\left\{ \frac{\sum_{e\in E}\ell_{e}\left(\sqrt{b-W_{e,\max}}-\sqrt{a-W_{e,\min}}\right)}{\pi}-\frac{\left|E\right|+\left|V^{\circ}\right|}{\mathcal{L}_{E}},0\right\} \leq\mu\bigl(\left(a,b\right)\bigl)\label{Bound_Segments}\\
 & \leq\mu\bigl(\left[a,b\right]\bigl)\leq\frac{\sum_{e\in E}\ell_{e}\left(\sqrt{b-W_{e,\min}}-\sqrt{a-W_{e,\max}}\right)}{\pi\mathcal{L}_{E}}+\frac{\left|E\right|+\left|V^{\circ}\right|}{\mathcal{L}_{E}}\,.\nonumber 
\end{align}
\end{enumerate}
\end{theorem}

\begin{corollary}\label{DOS_regular}
The density of states $\mu$ in the case of $\delta$-type vertex conditions is a locally finite and regular measure.
\end{corollary}

To prove Theorem \ref{DOS_bounds}, we use two results from \cite{Introduction_to_Quantum_Graphs}, a standard result on Sturm-Liouville operators, and a lemma we prove here.

\begin{lemma}\label{Eigenvelues_to_Infinity}\cite[Thm.\ 3.1.1]{Introduction_to_Quantum_Graphs}
Let $\Gamma$ be a compact quantum graph equipped with a self-adjoint Schr\"odinger operator $\mathcal{H}_{\Gamma}$. Then the resolvent operator $\left(\mathcal{H}_{\Gamma} - i\right)^{-1}$ is compact on $L^{2}(\Gamma)$. Consequently, $\mathcal{H}_{\Gamma}$ admits a complete orthonormal system of eigenfunctions, and its spectrum $\sigma(\mathcal{H}_{\Gamma})$ consists solely of isolated eigenvalues $\lambda_n$ of finite multiplicity, with $\lambda_n \to \infty$ as $n \to \infty$.
\end{lemma}

We denote by $\lambda_n(\Gamma)$ the (non-decreasing) sequence of eigenvalues of $\Gamma$, counted with multiplicities.

\begin{lemma}\label{Moves_of_Eigenvalues} \cite[Thm.\ 3.1.8]{Introduction_to_Quantum_Graphs}
Let $\Gamma_{\alpha'}$ be a compact quantum graph with $\delta$-type vertex conditions, obtained from the graph $\Gamma_{\alpha}$ by changing the coefficient of the condition at vertex $\upsilon$ from $\alpha$ to $\alpha'$. If $-\infty < \alpha < \alpha' \leq \infty$, then
\[
\lambda_n(\Gamma_{\alpha}) \leq \lambda_n(\Gamma_{\alpha'}) \leq \lambda_{n+1}(\Gamma_{\alpha})\,,
\]
where the eigenvalues are counted with multiplicities, and the notation follows that of Lemma~\ref{Eigenvelues_to_Infinity}.
\end{lemma}

The following is a standard result in the theory of Sturm-Liouville operators.

\begin{lemma}\label{Bound_Sturm_Liouville} (see, e.g., \cite{zettl} or \cite[Prop.\ 7.14]{Sturm_Liouville} for the precise statement)
Consider the system
\[
-\frac{d^{2}}{dx^{2}}u(t)+q(t)u(t)=\lambda u(t), \quad u(0)=u(L)=0
\]
on the interval $\left[0,L\right]$, with $L > 0$ and $q \in C\left([0,L]\right)$.  
Then the $n$-th eigenvalue $\lambda_n$ of this system satisfies the bounds
\begin{equation}
\frac{\pi^{2}n^{2}}{L^{2}}+q_{\min} \leq \lambda_n \leq \frac{\pi^{2}n^{2}}{L^{2}}+q_{\max},
\end{equation}
where
\begin{equation}
q_{\min} = \min_{t \in [0,L]} q(t), \qquad q_{\max} = \max_{t \in [0,L]} q(t).
\end{equation}
\end{lemma}

\begin{lemma}\label{Bound_Spectrum_From_Below}
Let $\Gamma=(V,E)$ be a compact quantum graph equipped with a Schr\"odinger operator $\mathcal{H}_{\Gamma}$. Then $\mathcal{H}_{\Gamma}$ is bounded from below, that is, its spectrum is bounded from below. Moreover, this lower bound can be chosen uniformly for $\Gamma$ and for all of its covering graphs.
\end{lemma}

\begin{proof}[Proof]
Let $\mathcal{H}_{0} = -\frac{d^{2}}{dx^{2}}$ denote the standard Schr\"odinger operator on $\Gamma$. As shown in Corollary~10 of~\cite{Basic_inf}, the operator $\mathcal{H}_{0}$ is bounded from below; more precisely, there exists a constant $C' \geq 0$ such that
\[
\mathcal{H}_{0} \geq -C' \,\mathrm{Id}.
\]
The constant $C'$ depends only on the finite set of edge lengths and on the vertex conditions of the graph.

Since the multiplication operator by the potential $W$ satisfies $\|W\|_{\mathrm{op}} = \|W\|_{\infty}$, Theorem~9.1 of \cite{Perturbation} guarantees, via standard perturbation theory, the corresponding result for the perturbed operator.

Finally, the lower bound is uniform over all covering graphs of $\Gamma$,
because the potential values, the edge lengths, and the vertex conditions are
preserved under coverings.
\end{proof}

\medskip

We can now prove Theorem \ref{DOS_bounds}.

\begin{proof}[Proof of Theorem~\ref{DOS_bounds}]
Let $\Gamma_{n}=\left(V_{n},E_{n}, \ell_{n}, \mathcal{H}_{\Gamma_{n}}\right)$ be a sequence of $m_{n}$-covers of $\Gamma=\left(V,E, \ell, \mathcal{H}_{\Gamma}\right)$ with the covering maps $\Xi_{n}:\Gamma_{n}\rightarrow\Gamma$, where $m_{n}\rightarrow\infty$ as $n\rightarrow\infty$, as in Theorem \ref{ESM_limit}(i), such that the girth tends to infinity. Let $x > W_{\max}$.
For each $n$, let $\Gamma_{n,D}$ be the graph obtained from $\Gamma_n$ by imposing Dirichlet conditions at all vertices $\upsilon \in V_n$. Let $\{\lambda_{n, \ell}\}_{\ell \in \mathbb{N}}$ be the eigenvalues of $\mathcal{H}_{\Gamma_n}$, and let $\{\lambda_{\ell}(\Gamma_{n,D})\}_{\ell \in \mathbb{N}}$ be the eigenvalues of $\mathcal{H}_{\Gamma_{n,D}}$.

For every $e \in E_n$, let $\mathfrak{a}_{n,e}$ be the largest non-negative integer such that
\[
\frac{\mathfrak{a}_{n,e}^2 \pi^2}{\ell_e^2} + W_{e,\min} \leq x,
\]
and define $\mathfrak{b}_{n,e}$ similarly using $W_{e,\max}$ instead. Then
\[
\frac{\ell_{e}}{\pi}\sqrt{x-W_{e,\min}}-1<\mathfrak{a}_{n,e}\leq\frac{\ell_{e}}{\pi}\sqrt{x-W_{e,\min}}
\]
\[
\frac{\ell_{e}}{\pi}\sqrt{x-W_{e,\max}}-1<\mathfrak{b}_{n,e}\leq\frac{\ell_{e}}{\pi}\sqrt{x-W_{e,\max}}\,.
\]
From Lemma~\ref{Bound_Sturm_Liouville} it follows that
\begin{align*}
& \sum_{e\in E_{n}}\frac{\ell_{e}}{\pi}\sqrt{x-W_{e,\max}}-\left|E_{n}\right|<\sum_{e\in E_{n}}\mathfrak{b}_{n,e}\leq\#\left\{ \ell\in\mathbb{N}\mid\lambda_{\ell}\left(\Gamma_{n,D}\right)\leq x\right\} \leq\\
& \leq\sum_{e\in E_{n}}\mathfrak{a}_{n,e}\leq\sum_{e\in E_{n}}\frac{\ell_{e}}{\pi}\sqrt{x-W_{e,\min}}\,.
\end{align*}
By Lemma~\ref{Moves_of_Eigenvalues}, after restoring the original vertex conditions at the $m_n |V^\circ|$ interior vertices, the cumulative distribution function $\nu_{n}\bigl((-\infty, x]\bigl)$ of the eigenvalue counting measure satisfies
\[
\frac{\sum_{e\in E}\ell_{e}\sqrt{x-W_{e,\max}}}{\pi\mathcal{L}_{E}}-\frac{\left|E\right|}{\mathcal{L}_{E}}<\nu_{n}\bigl(\left(-\infty,x\right]\bigl)\leq\frac{\sum_{e\in E}\ell_{e}\sqrt{x-W_{e,\min}}}{\pi\mathcal{L}_{E}}+\frac{\left|V^{\circ}\right|}{\mathcal{L}_{E}}\,.
\]
Similarly, for the open interval:
\[
\frac{\sum_{e\in E}\ell_{e}\sqrt{x-W_{e,\max}}}{\pi\mathcal{L}_{E}}-\frac{\left|E\right|}{\mathcal{L}_{E}}\leq\nu_{n}\bigl(\left(-\infty,x\right)\bigl)<\frac{\sum_{e\in E}\ell_{e}\sqrt{x-W_{e,\min}}}{\pi\mathcal{L}_{E}}+\frac{\left|V^{\circ}\right|}{\mathcal{L}_{E}}\,.
\]

By Lemma~\ref{Bound_Spectrum_From_Below}, there exists a constant $C \geq 0$
such that for every $n$ the spectrum of $\mathcal{H}_{n}$ is contained in
$[-C,\infty)$. Consequently,
\begin{align*}
\nu_{n}\bigl((-\infty,x)\bigr) &= \nu_{n}\bigl((-C,x)\bigr),\\
\nu_{n}\bigl((-\infty,x]\bigr) &= \nu_{n}\bigl([-C,x]\bigr),
\end{align*}
and similarly,
\begin{align*}
\mu\bigl((-\infty,x)\bigr) &= \mu\bigl((-C,x)\bigr),\\
\mu\bigl((-\infty,x]\bigr) &= \mu\bigl([-C,x]\bigr).
\end{align*}

By the vague convergence established in Theorem~\ref{ESM_limit}(ii), we have,
for every $n$,
\[
\mu\bigl((-C,x)\bigr)
\;\leq\;
\nu_{n}\bigl((-C,x)\bigr)
\;\leq\;
\nu_{n}\bigl([-C,x]\bigr)
\;\leq\;
\mu\bigl([-C,x]\bigr).
\]

Therefore, for every $\varepsilon>0$,
\begin{align*}
&\limsup_{n\rightarrow\infty}\nu_{n}\bigl([-C+\varepsilon,x-\varepsilon]\bigr)
\leq
\mu\bigl([-C+\varepsilon,x-\varepsilon]\bigr)
\leq
\mu\bigl((-C,x)\bigr)\\
&\leq
\mu\bigl([-C,x]\bigr)
\leq
\mu\bigl((-C-\varepsilon,x+\varepsilon)\bigr)
\leq
\liminf_{n\rightarrow\infty}\nu_{n}\bigl((-C-\varepsilon,x+\varepsilon)\bigr).
\end{align*}

Finally, letting $\varepsilon \downarrow 0$ and using the monotone continuity
of measures yields inequality~\eqref{Bound_commulative}.

Now~\eqref{Asymptote_commulative} follows immediately by expanding $\sqrt{1-\frac{W_{e,\max}}{x}}$ and $\sqrt{1-\frac{W_{e,\min}}{x}}$ into a Taylor series.

For any $W_{\max} \leq a \leq b < \infty$, we obtain
\begin{align*}
 & \frac{\sum_{e\in E}\ell_{e}\left(\sqrt{b-W_{e,\max}}-\sqrt{a-W_{e,\min}}\right)}{\pi}-\frac{\left|E\right|+\left|V^{\circ}\right|}{\mathcal{L}_{E}}\leq\mu\bigl(\left(a,b\right)\bigl)=\\
 & =\mu\bigl(\left(-\infty,b\right)\bigl)-\mu\bigl(\left(-\infty,a\right]\bigl)\leq\frac{\sum_{e\in E}\ell_{e}\left(\sqrt{b-W_{e,\min}}-\sqrt{a-W_{e,\max}}\right)}{\pi\mathcal{L}_{E}}+\frac{\left|E\right|+\left|V^{\circ}\right|}{\mathcal{L}_{E}}\,.
\end{align*}
The same proof applies to closed intervals $[a,b]$, establishing~\eqref{Bound_Segments}.
\end{proof}

\section{PROOF OF THE MAIN THEOREMS}

We begin with an important observation.

\begin{observation}\label{Observation_in_Q-Aomoto}

In the Q-Aomoto set there are two types of connected components:
\begin{enumerate}[label=\roman*.]
\item A connected set of vertices together with the edges emanating from them.
\item Edges $e=\left(\upsilon,u\right)\in E_{\lambda}\left(\Gamma\right)$
such that $g_{e}\left(\ell_{e}\right)=0$ (recall \eqref{f_and_g}), where $\upsilon,u\notin X_{\lambda}\left(\Gamma\right)$. Note that by (\ref{Adjustment_in_the_other_direction}), $g_{e}\left(\ell_{e}\right)=0$
if and only if $g_{\check{e}}\left(\ell_{e}\right)=0$.
\end{enumerate}
\end{observation}
An example for Q-Aomoto set with these two types of connected components is shown in Figure~\ref{fig:Q-Aomoto_Example}.

\begin{minipage}{0.9\textwidth}
\begin{center}
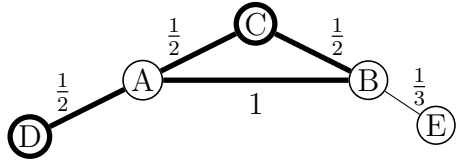

\begin{tikzpicture}[scale=1.5]

\begin{scope}[shift={(0,2.5)}]
\tikzset{
  vertex/.style={circle, draw, line width=0.6pt, minimum size=4mm, inner sep=1pt},
  thickvertex/.style={circle, draw, line width=2pt, minimum size=4mm, inner sep=1pt}
}

\node[vertex] (A) at (0,0) {A};
\node[vertex] (B) at (2,0) {B};
\node[thickvertex] (C) at (1,0.5) {C}; 
\node[thickvertex] (D) at (-1, -0.5) {D};
\node[vertex] (E) at (2.6, -0.4) {E};

\draw[line width=2pt] (A) -- node[below] {1} (B);
\draw[line width=2pt] (A) -- node[left, xshift=-1 pt, yshift=6 pt] {$\frac{1}{2}$} (C);
\draw[line width=2pt] (B) -- node[right, xshift=1 pt, yshift=6 pt] {$\frac{1}{2}$} (C);
\draw[line width=2pt] (D) -- node[left, xshift=-1 pt, yshift=6 pt] {$\frac{1}{2}$} (A);

\draw[line width=0.4pt] (E) -- node[right, xshift=-1 pt, yshift=6 pt] {$\frac{1}{3}$} (B);
\end{scope}
\end{tikzpicture}

\vspace{1em}
\captionof{figure}{A quantum graph $\Gamma$, in which the Q-Aomoto set corresponding to $\lambda=\pi^{2}$ is highlighted. The graph is equipped with the standard Schr\"odinger operator and Kirchhoff boundary conditions at all vertices. Edge lengths are indicated along the edges. The vertex $D$ together with its incident edge, and the vertex $C$ together with its incident edges, form two connected components of the first type described in Observation~\ref{Observation_in_Q-Aomoto}. The edge $(A,B)$ forms a connected component of the second type. In this example, $\partial X_{\pi^{2}}(\Gamma)=\{A,B\}$.}
\label{fig:Q-Aomoto_Example}
\end{center}
\end{minipage}

\begin{proof}[Proof of Theorem~\ref{Q-Aomoto_first_theorem}(i)-(ii)]
The results are immediate for the second type of Q-Aomoto components (according to Observation~\ref{Observation_in_Q-Aomoto}).
For the first type, consider a connected component $X \subseteq \Xi^{-1}(X_{\lambda}(\Gamma))$. Using a countable linear combination of eigenfunctions, we can construct a function $h$ such that $h(\upsilon) \neq 0$ for every $\upsilon \in V(X)$ and such that $h|_{e}$ is not identically zero for every $e \in E(X)$. Impose Dirichlet boundary conditions on the boundary $\partial X$ of the subgraph $X \cup \partial X \cup \partial_{\infty}X$. Then, by Lemma \ref{Lemma_From_Derived_Graph}(i), the restrictions of the eigenfunctions to $V(X)$ lie in a one-dimensional subspace of $\ell^2(V(X))$. If $\Xi(X)$ contained a cycle, then $X$ would be infinite and periodic, contradicting the symmetry of the subset and the invariance of $h$ under translations.

We have thus proved Theorem \ref{Q-Aomoto_first_theorem}(i), and in particular, shown that $X$ is finite.
Finally, Lemma \ref{Lemma_From_Derived_Graph}(ii) yields Theorem \ref{Q-Aomoto_first_theorem}(ii).
\end{proof}

Recall that a \emph{pure cycle} in a graph is one in which every vertex is of degree $2$ except perhaps one through which the cycle is connected to the rest of the graph. The next lemma says that if an eigenfunction's support includes some vertices in the lift of a pure cycle, then it must vanish on the vertex connecting it to the graph.

\begin{lemma}\label{Remark_Pure_Cycle}
Let $L\subseteq\Gamma$ be a pure cycle and let $u$ be the vertex through which $L$ is connected to the rest of the graph $\Gamma$. Suppose that for every vertex $\upsilon \in V(L)$, $\alpha_{\upsilon}<\infty$. Let $\lambda \in \sigma_{\mathfrak{p}}(\mathcal{H}_{\mathcal{T}})$ and assume that $X_{\lambda}(\Gamma) \cap L \neq \emptyset$. Then necessarily, $u \in \partial X_{\lambda}(\Gamma)$.
\end{lemma}

\begin{proof}[Proof]
From the $\delta$-type conditions~\eqref{Delta_Type}, it follows that all edges of $L$ lie in $E_{\lambda}(\Gamma)$. Suppose for the sake of contradiction that $u \in V_{\lambda}(\Gamma)$. Then there exists an eigenfunction $h \in \ker(\lambda - \mathcal{H}_{\mathcal{T}})$ which does not vanish on an infinite connected component $X \subseteq \mathcal{T}$, possibly omitting some vertices of degree 2 from $\Xi^{-1}\left(L\right)$ in the middle. Then the inductive argument from the proof of Lemma~\ref{Lemma_From_Derived_Graph} can be applied here.

Indeed, for every $\upsilon \in V(L) \setminus \{u\}$ and each $\tilde{\upsilon} \in \Xi^{-1}(\upsilon)$ with $h(\tilde{\upsilon}) = 0$, let $e_{1}, e_{2}=\left(\tilde{\upsilon},w\right) \in E_{\tilde{\upsilon}}^-$. From (\ref{Delta_Type}) we have
\[
h_{e_{1}}^{\prime}\left(\tilde{\upsilon}\right)=-h_{e_{2}}^{\prime}\left(\tilde{\upsilon}\right).
\]
This allows us to determine the value $\gamma(h)(\tilde{\upsilon}_{e_2})$ from $\gamma(h)(\tilde{\upsilon}_{e_1})$, where $\gamma(h)(\tilde{\upsilon}_{p}) = 0$ (recall the derived graph from the previous section). Subsequently, applying the propagation relations from~\eqref{Matrix_f_g}, we can determine $\gamma(h)(w_p)$ and $\gamma(h)(w_{e_2})$, and continue inductively along the cycle.

By initiating the ``induction'' from a lift $\tilde{u} \in \Xi^{-1}(u) \cap X$, we conclude that the restriction of eigenfunctions to $X$ must lie in a one-dimensional subspace of $L^2(X)$. However, the inherent symmetry of the cycle and translation invariance within $L$ contradict this conclusion, implying that $u \in \partial X_{\lambda}(\Gamma)$ as claimed.
\end{proof}

Based on Theorem~\ref{Q-Aomoto_first_theorem}\,(i)--(ii), we now introduce notation that will be useful in the proofs of the subsequent theorems.

\begin{notation}\label{Notation_h_i_s}
Fix $\lambda \in \sigma_{\mathfrak{p}}(\mathcal{H}_{\mathcal{T}})$. Let $T_{1}, \ldots, T_{m}$ be the Q-Aomoto connected components of $\Gamma = (V, E, \ell, \mathcal{H}_{\Gamma})$ associated with $\lambda$ (so $m=ccX_\lambda(\Gamma)$ - the number of connected components of the Q-Aomoto set). Since $X_\lambda(\Gamma)$ is acyclic, its connected components are trees so they can be identified with the connected components of the lift of $X_\lambda(\Gamma)$ to $\mathcal{T}$. By Theorem~\ref{Q-Aomoto_first_theorem}(ii) the restriction of the space of $\lambda$-eigenfunctions to each such component is one-dimensional, so we may define $h_{1}, \ldots, h_{m}$ to be the normalized functions obtained from the restrictions of the eigenfunctions onto the respective components. These functions are also normalized eigenfunctions of the subgraph
\[
X_{\lambda}(\Gamma) \cup \partial X_{\lambda}(\Gamma) \cup \partial_{\infty}X_{\lambda}(\Gamma)\,,
\]
where now Dirichlet boundary conditions are imposed on $\partial X_{\lambda}(\Gamma)$. Because the potential is real valued, we may assume that the functions $h_1, \ldots, h_m$ are real valued and normalized. Let $S_{i}$ denote the set of disjoint copies of the subgraph $T_{i}$ in the universal cover $\mathcal{T} = (\mathcal{V}, \mathcal{E}, \ell_{\mathcal{T}}, \mathcal{H}_{\mathcal{T}})$. For each $s \in S_{i}$, we define a function $h_{i,s} \in H^{2}(\mathcal{T})$ by
\begin{equation}
h_{i,s}\left(x\right)=\begin{cases}
h_{i}\left(\Xi\left(x\right)\right) & x\in s\\
0 & \textrm{otherwise}.
\end{cases}\,.\label{h_i_s}
\end{equation}
Then every $h\in\ker\left(\lambda-\mathcal{H}_{\mathcal{T}}\right)$ can be decomposed as
\begin{equation}
h=\sum_{i=1}^{m}\sum_{s\in S_{i}}\beta_{s}h_{s},\ \beta_{s}\in\mathbb{C}\,.\label{h_with_alpha}
\end{equation}
Note that the functions $h_{s}$ do not necessarily belong to $\text{Dom}\left(\mathcal{H}_{\mathcal{T}}\right)$, as they are not required to satisfy the vertex conditions at the boundary $\partial X_{\lambda}\left(\Gamma\right)$.\\
\end{notation}

We now construct a bipartite discrete graph $\Gamma^{\prime}=\left(V^{\prime},E^{\prime},a^{\prime},b^{\prime}\right)$ (which is not the derived graph from Definition \ref{Derived_Graph}). The construction follows a procedure similar to that introduced by Banks, Garza-Vargas, and Mukherjee in~\cite[Section 5.1]{Garza-Vargas} for discrete graphs. We define the vertex set as $V^{\prime} = U \cup \partial X_{\lambda}\left(\Gamma\right)$, where $U = \left\{ t_{1}, \dots, t_{m} \right\}$ and each $t_{i}$ corresponds to the Q-Aomoto component $T_{i}$. For each $i \in \left[m\right]$, the vertex $t_{i}$ is referred to as the \textit{representative vertex} of $T_{i}$. The edge set $E^{\prime}$ consists of edges connecting each $t_{i}$ to any vertex $\upsilon \in \partial X_{\lambda}\left(\Gamma\right)$ that is adjacent to $T_{i}$ in the original graph $\Gamma$. Note that a connected component $T_i$ might be a single edge (without its connected vertices). Such an edge is represented by a vertex in $\Gamma^{\prime}$.

For each $i \in [m]$ and each $\upsilon \in \partial X_{\lambda}(\Gamma)$ that is adjacent to the subtree $T_{i}$ at the vertex $u_{i} \in V(T_{i})$, we define the weight
$a_{(t_{i},\upsilon)}^{\prime} := \left(h_{i}\right)_{(u_{i},\upsilon)}^{\prime}(\upsilon)$.

We set $b^{\prime} \equiv 0$. If a vertex is connected to a Q-Aomoto component via multiple edges in the original graph, then multiple corresponding edges will appear in $\Gamma^{\prime}$, resulting in possible edge multiplicities. That is, the bipartite graph $\Gamma^{\prime}$ may include parallel edges between a vertex in $\partial X_{\lambda}(\Gamma)$ and a representative vertex $t_i$ of a Q-Aomoto component. An illustrative example of such a corresponding graph is presented in Figure~\ref{fig:corresponding_graph}.

\begin{minipage}{0.9\textwidth}
\begin{center}
\begin{tikzpicture}[scale=1.5]

\begin{scope}[shift={(0,2.5)}]
\tikzset{
  vertex/.style={circle, draw, line width=0.6pt, minimum size=4mm, inner sep=1pt},
  thickvertex/.style={circle, draw, line width=2pt, minimum size=4mm, inner sep=1pt}
}

\node[vertex] (A) at (0,0) {A};
\node[vertex] (B) at (2,0) {B};
\node[thickvertex] (C) at (1,0.5) {C}; 
\node[thickvertex] (D) at (-1, -0.5) {D};
\node[vertex] (E) at (2.6, -0.4) {E};

\draw[line width=2pt] (A) -- node[below] {1} (B);
\draw[line width=2pt] (A) -- node[left, xshift=-1 pt, yshift=6 pt] {$\frac{1}{2}$} (C);
\draw[line width=2pt] (B) -- node[right, xshift=1 pt, yshift=6 pt] {$\frac{1}{2}$} (C);
\draw[line width=2pt] (D) -- node[left, xshift=-1 pt, yshift=6 pt] {$\frac{1}{2}$} (A);

\draw[line width=0.4pt] (E) -- node[right, xshift=-1 pt, yshift=6 pt] {$\frac{1}{3}$} (B);
\end{scope}

\begin{scope}[shift={(0,0)}]
\node[circle, draw, minimum size=3mm, inner sep=1pt] (A) at (0-0.1,0) {};
\node[circle, draw, minimum size=3mm, inner sep=1pt] (B) at (3*0.75-0.1,0) {};
\node[circle, draw, minimum size=3mm, inner sep=1pt,  line width=2pt] (C) at (1.5*0.75-0.1,0.75) {};
\node[circle, draw, minimum size=3mm, inner sep=1pt,  line width=2pt] (D) at (-1.5*0.75-0.1, -0.75) {};
\node[circle, draw, minimum size=3mm, inner sep=1pt,  line width=2pt] (M) at (1.5*0.75-0.1,-0.2) {}; 

\draw[line width=1pt] (A) -- (D);
\draw[line width=1pt] (A) -- (C);
\draw[line width=1pt] (C) -- (B);
\draw[line width=1pt] (A) -- (M);
\draw[line width=1pt] (B) -- (M);
\end{scope}

\draw[->, thick, line width=1.5pt, shorten >=5pt] (1,1.9) -- (1,0.9);

\end{tikzpicture}

\vspace{1em}

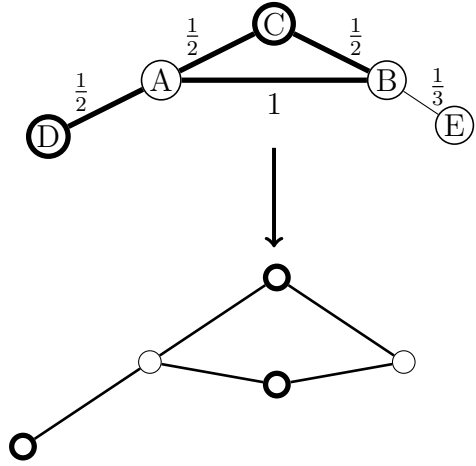
\captionof{figure}{The quantum graph from Figure~\ref{fig:Q-Aomoto_Example} and its associated discrete graph corresponding to the eigenvalue $\lambda = \pi^2$. The Q-Aomoto set for $\lambda$ is highlighted. In the discrete graph below, the vertices corresponding to the connected components of the Q-Aomoto set are also highlighted. Note that the edge $\left(A, B\right)$, representing a Q-Aomoto component of the second type described in Observation~\ref{Observation_in_Q-Aomoto}, is represented as a vertex in the discrete model.}
\label{fig:corresponding_graph}
\end{center}
\end{minipage}

\vspace{1em}
Let $\mathcal{T}^{\prime}$ denote the universal cover of $\Gamma^{\prime}$, let $\Xi^{\prime} : \mathcal{T}^{\prime} \rightarrow \Gamma^{\prime}$ be the associated covering map, and let $A_{\mathcal{T}^{\prime}}$ be the Jacobi operator defined by the parameters $\left\{a',b'\right\}$. We define a map 
\[
\varphi : \ker\left(\lambda - \mathcal{H}_{\mathcal{T}}\right) \overset{\sim}{\longrightarrow} \ker\left(A_{\mathcal{T}^{\prime}}\right)
\]
by
\begin{equation}
\varphi(h) := \varphi\left( \sum_{i=1}^{m} \sum_{s \in S_{i}} \beta_{s} h_{s} \right) := \sum_{i=1}^{m} \sum_{s \in S_{i}} \beta_{s} \delta_{u_{s}} =: \tilde{h}\,, \label{Phi}
\end{equation}
for every $h \in \ker\left(\lambda - \mathcal{H}_{\mathcal{T}}\right)$, where $u_{s}$ denotes the vertex in $\mathcal{T}^{\prime}$ corresponding to the Q-Aomoto connected component $s$, and $\delta_{u_{s}}$ is the indicator function at $u_{s}$.

Following the strategy in \cite{Garza-Vargas} we will first show that $\varphi$ defines an isometric embedding, and subsequently establish that it is an isomorphism. 

\begin{proposition}\label{phi_Isometric_Inclusion}
The map $\varphi$ defined in~\eqref{Phi} is an isometric inclusion of $\ker\left(\lambda - \mathcal{H}_{\mathcal{T}}\right)$ into $\ker\left(A_{\mathcal{T}^{\prime}}\right)$.
\end{proposition}

\begin{proof}[Proof]
The isometry follows immediately from the normalization of each function $h_{s}$. It remains to verify that $\varphi(h) \in \ker\left(A_{\mathcal{T}^{\prime}}\right)$. Since $\mathcal{T}^{\prime}$ is a bipartite graph and $b^{\prime} \equiv 0$, it suffices to show that
\[
\left(A_{\mathcal{T}^{\prime}} \varphi(h)\right)(\upsilon) = 0
\]
for every $\upsilon \in \Xi^{\prime -1} \left( \partial X_{\lambda}(\Gamma) \right)$. Indeed, in this case we have
\[
\left(A_{\mathcal{T}^{\prime}} \varphi(h)\right)(\upsilon) = \sum_{e \in E^{-}_{\upsilon}} h_{e}^{\prime}(\upsilon) = \alpha_{\upsilon} h(\upsilon) = 0
\]
as required.
\end{proof}

\begin{proposition}\label{Aomoto_Preservation}
Recall Definition~\ref{Q-Aomoto_def} of the descrete Aomoto set. Let $G_{1}^{\prime},...,G_{n}^{\prime}$ denote the connected components of $\Gamma^{\prime}$, and define $U_{j} := G_{j}^{\prime} \cap U$ for each $j = 1, \ldots, n$. Then the Aomoto set of $G_j$ corresponding to the eigenvalue $0$ satisfies 
\[
X_{0}\left(G_{j}^{\prime}\right)=U_{j}\,.
\]
\end{proposition}

\begin{proof}[Proof]
Denote by $T_1, \ldots, T_m$ the Aomoto components of $\Gamma$ such that their corresponding representative vertices $U_j = \{t_1, \ldots, t_m\}$ are in $G_j'$. For each $i$, there exists a function $h \in \ker(\lambda - \mathcal{H}_{\mathcal{T}})$ such that $\varphi(h)(s_i) \neq 0$ for some $s_i \in \Xi^{\prime -1}(t_i)$. By Proposition~\ref{phi_Isometric_Inclusion}, we have $\varphi(h) \in \ker(A_{\mathcal{T}'})$, and thus $U_j \subseteq X_0(G_j')$.

According to~\cite[Theorem 3.1(i)]{Garza-Vargas}, the Aomoto components of a discrete graph are acyclic, and therefore $X_0(G_j')$ forms a forest in $G_j'$. Suppose, for contradiction, that some vertices in $X_0(G_j')$ are connected. Then there exists a tree $T_i \subseteq X_0(G_j')$ with more than one vertex, and a function $\varphi(h) \in \ker(A_{\mathcal{T}'})$ such that $\varphi(h)(u) \neq 0$ for every $u$ in some lift $S \subseteq \Xi^{\prime -1}(T_i)$ of $T_i$.

However, since $A_{\mathcal{T}'}$ has zero potential, we must have $\left(A_{\mathcal{T}'}\varphi(h)\right)(\tilde{\upsilon}) \neq 0$ for a leaf $\tilde{\upsilon}$ in $S$, yielding a contradiction. Therefore, $X_0(G_j')$ must be an independent set. From the construction of $G_j'$, it follows that $U_j$ is a maximal independent set in $G_j'$, and thus $X_0(G_j') = U_j$.
\end{proof}

\begin{proposition}\label{Isomorthism_Proposition}
The map $h\mapsto\tilde{h}$ is an isometric isomorphism between $\ker\left(\lambda-\mathcal{H}_{\mathcal{T}}\right)$
and $\ker\left(A_{\mathcal{T}^{\prime}}\right)$.
\end{proposition}

\begin{proof}[Proof]
We want to prove that $\varphi$ is onto. Indeed, for every $\eta\in\ker\left(A_{\mathcal{T}^{\prime}}\right)$,
from Proposition \ref{Aomoto_Preservation}, there are $\beta_{s}$ such that
\[
\eta=\sum_{i=1}^{m}\sum_{s\in S_{i}}\beta_{s}\delta_{u_{s}}
\]
with the notations of Notation \ref{Notation_h_i_s}. Then
\[
\varphi\left(\sum_{i=1}^{m}\sum_{s\in S_{i}}\beta_{s}h_{s}\right)=\eta\,.
\]
\end{proof}

\begin{lemma}\label{From_Eigenfunctions_to_mu_1}
\begin{enumerate}[label=\roman*.]
\item Let $G = (V, E, a, b)$ be a finite weighted discrete graph, $\mathcal{T}$ its universal cover with covering map $\Xi: \mathcal{T} \to G$, and let $\lambda \in \sigma_{\mathfrak{p}}(A_{\mathcal{T}})$. Let $\mathcal{B}$ be an orthonormal basis of $\ker(\lambda - A_{\mathcal{T}})$. Then for every $u \in V$ and every $\tilde{u} \in \Xi^{-1}(u)$, the spectral measure associated with $\delta_{\tilde{u}}$, $\mu_u$, is independent of the choice of $\tilde{u} \in \Xi^{-1}(u)$ and satisfies
\[
\mu_u(\{\lambda\}) = \sum_{\eta \in \mathcal{B}} \left| \eta(\tilde{u}) \right|^2\,.
\]

\item Let $\Gamma = (V, E, \ell, \mathcal{H}_{\Gamma})$ be a compact quantum graph with $\delta$-type vertex conditions and a Schr\"odinger-type differential operator $\mathcal{H}_{\Gamma}$. Let $\mathcal{T}$ be its universal cover with covering map $\Xi: \mathcal{T} \to \Gamma$, and let $\lambda \in \sigma_{\mathfrak{p}}(\mathcal{H}_{\mathcal{T}})$. Let $\mathcal{B}$ be an orthonormal basis of $\ker(\mathcal{H}_{\mathcal{T}} - \lambda)$. Then, for every $g \in L^2(\mathcal{T})$ with spectral measure $\mu_g$, we have
\[
\mu_g(\{\lambda\}) = \sum_{h \in \mathcal{B}} \left| \langle g, h \rangle \right|^2\,.
\]
\end{enumerate}
\end{lemma}

\begin{proof}[Proof]
Let $P_{\left\{ \lambda\right\} }$ denote the spectral projection onto the eigenspace corresponding to the eigenvalue \( \lambda \). Then, for the discrete case, we have:
\begin{align*}
& \mu_{u}\left(\left\{ \lambda\right\} \right)=\left\langle P_{\left\{ \lambda\right\} }\delta_{\tilde{u}},\delta_{\tilde{u}}\right\rangle =\left\Vert P_{\left\{ \lambda\right\} }\delta_{\tilde{u}}\right\Vert ^{2}=\sum_{\eta\in\mathcal{B}}\left|\left\langle P_{\left\{ \lambda\right\} }\delta_{\tilde{u}},\eta\right\rangle \right|^{2}=\\
& =\sum_{h\in\mathcal{B}}\left|\left\langle \delta_{\tilde{u}},P_{\left\{ \lambda\right\} }\eta\right\rangle \right|^{2}=\sum_{h\in\mathcal{B}}\left|\left\langle \delta_{\tilde{u}},\eta\right\rangle \right|^{2}=\sum_{h\in\mathcal{B}}\left|\eta\left(\tilde{u}\right)\right|^{2}\,.
\end{align*}
The independence from the choice of $\tilde{u}$ follows from the symmetry. 

For the quantum (continuum) case, one similarly obtains:
\[
\mu_{g}\left(\left\{ \lambda\right\} \right)=\sum_{h\in\mathcal{B}}\left|\left\langle P_{\left\{ \lambda\right\} }g,h\right\rangle \right|^{2}=\sum_{h\in\mathcal{B}}\left|\left\langle g,P_{\left\{ \lambda\right\} }h\right\rangle \right|^{2}=\sum_{h\in\mathcal{B}}\left|\left\langle g,h\right\rangle \right|^{2}\,.
\]
\end{proof}

\begin{lemma}\label{From_Eigenfunctions_to_mu_2}
Let \( \Gamma = (V, E, \ell, \mathcal{H}_{\Gamma}) \) be a compact quantum graph with \(\delta\)-type vertex conditions and a Schr\"odinger-type differential operator \( \mathcal{H}_{\Gamma} \). Let \( \mathcal{T} \) be its universal cover with covering map \( \Xi : \mathcal{T} \to \Gamma \), and let \( \lambda \in \sigma_{\mathfrak{p}}(\mathcal{H}_{\mathcal{T}}) \). 

Let \( \Gamma' = (V', E', a', b') \) be the corresponding bipartite discrete graph, and let \( T \subseteq \Gamma \) be a Q-Aomoto component with corresponding vertex \( t \in V(\Gamma') \). Denote by \( \mu_t' \) the spectral measure of the indicator function \( \delta_t \).

Let \( \{g_k\}_{k \in \mathbb{N}} \) be an orthonormal basis of \( L_{2}(T) = \bigoplus_{e \in T} L_{2}(e) \), and let \( \mu_{g_k} \) be the spectral measure associated with the function \( g_k \circ \Xi|_{\widetilde{T}} \), where \( \widetilde{T} \subseteq \mathcal{T} \) is a copy of \( T \). Then
\[
\sum_{k\in\mathbb{N}}\mu_{g_{k}}\left(\left\{ \lambda\right\} \right)=\mu_{t}^{\prime}\left(\left\{ 0\right\} \right)\,.
\]
\end{lemma}

\begin{proof}[Proof]
Let \( \mathcal{T}' \) be the universal cover of \( \Gamma' \) and let \( \tilde{t} \in \mathcal{T}' \) be the vertex corresponding to \( \widetilde{T} \). Then, since $\varphi$ is an isometry, we have that for every \( h \in \ker(\mathcal{H}_{\mathcal{T}} - \lambda) \) 
\[
|\varphi\left(h\right)\left(\tilde{t}\right)|^2 = \sum_{\tilde{e} \in E(\widetilde{T})} \int_{\tilde{e}} |h|^2 \, dx = \sum_{k \in \mathbb{N}} \left| \langle h, \tilde{g}_k \rangle \right|^2\,,
\]
where \( \tilde{g}_k := g_k \circ \Xi|_{\widetilde{T}} \) for every \( k \in \mathbb{N} \).

Using Lemma~\ref{From_Eigenfunctions_to_mu_1} we obtain
\[
\sum_{k \in \mathbb{N}} \mu_{g_k}(\{\lambda\}) 
= \sum_{k \in \mathbb{N}} \sum_{h \in \mathcal{B}} \left| \langle h, \tilde{g}_k \rangle \right|^2 
= \sum_{h \in \mathcal{B}} \sum_{k \in \mathbb{N}} \left| \langle h, \tilde{g}_k \rangle \right|^2 
= \sum_{h \in \mathcal{B}} \|h\|_{L^2(\widetilde{T})}^2 
= \sum_{h \in \mathcal{B}} |\varphi\left(h\right)(\tilde{t})|^2 
= \mu_t'(\{0\})\, ,
\]
since $\{\varphi(h)\}_{h \in \mathcal{B}}$ is a basis of $\ker\left( A_{\mathcal{T}} \right)$.
\end{proof}

Now we can prove Theorem \ref{Q-Aomoto_first_theorem}(iii).

\begin{proof}[Proof of Theorem~\ref{Q-Aomoto_first_theorem}(iii)]
In~\cite[Observation 5.5]{Garza-Vargas}, Banks, Garza-Vargas, and Mukherjee considered a similar setting and proved that for the spectral measure $\mu_t'$ of a vertex $t \in U$ (where $U$ is the set of vertices corresponding to Q-Aomoto components), the index of the corresponding graph satisfies:
\[
I_{0}\left(\Gamma^{\prime}\right)=cc\left(X_{0}\left(\Gamma^{\prime} \right)\right)-\left|\partial X_{0}\left(\Gamma^{\prime}\right)\right|=\sum_{t\in U}\mu_{t}^{\prime}\left(\left\{ 0\right\} \right)\,.
\]
However, note that $cc\left(X_{0}\left(\Gamma^{\prime}\right)\right)=cc\left(X_{\lambda}\left(\Gamma\right)\right)$
and $\left|\partial X_{0}\left(\Gamma^{\prime}\right)\right|=\left|\partial X_{\lambda}\left(\Gamma\right)\right|$.

We want to compute the density of states (recall Definition \ref{The_measures}). Thus, let $X$ be a connected fundamental set of $\mathcal{T}$. Let $\{g_l'\}_{l \in \mathbb{N}}$ be an orthonormal basis of
$L_{2}\left(\Gamma\backslash E_{\lambda}\left(\Gamma\right)\right)$, and for each $\ell \in \mathbb{N}$, let $\mu_{g_\ell'}$ denote the spectral measure of its lift $\tilde{g}_\ell'$ to $X$ (this is naturally defined since there is an obvious identification between $E(\Gamma)$ and $E(X)$). Then, by Lemma~\ref{From_Eigenfunctions_to_mu_1}(ii), we have:
\[
\mu_{g_{\ell}^{\prime}}\left(\left\{ \lambda\right\} \right)=\sum_{h\in\mathcal{B}}\left|\left\langle \tilde{g}_{\ell}^{\prime},h\right\rangle \right|^{2}=0\,,
\]
since $\operatorname{supp}(h) \subseteq E_\lambda(\Gamma)$ for all $h \in \mathcal{B}$, and therefore $\langle \tilde{g}_\ell', h \rangle = 0$ for every $\ell$.

Now denote by $\mathcal{U}$ the collection of all Q-Aomoto connected components. For each $T \in \mathcal{U}$, let $\{g_k^T\}_{k \in \mathbb{N}}$ be an orthonormal basis of $L_{2}(T)$, and denote by $\mu_{g_k^T}$ the spectral measure of the lift of $g_k^T$ to $X$. Then, applying \eqref{Q_Den_of_States} together with Lemma~\ref{From_Eigenfunctions_to_mu_2}, we obtain
\begin{align*}
& \mathcal{L}_{E}\mu\left(\left\{ \lambda\right\} \right)=\sum_{T\in\mathcal{U}}\sum_{k\in\mathbb{N}}\mu_{g_{k}^{T}}\left(\left\{ \lambda\right\} \right)+\sum_{\ell\in\mathbb{N}}\mu_{g_{\ell}^{\prime}}\left(\left\{ \lambda\right\} \right)=\\
& =\sum_{T\in\mathcal{U}}\sum_{k\in\mathbb{N}}\mu_{g_{k}^{T}}\left(\left\{ \lambda\right\} \right)=\sum_{t\in U}\mu_{t}^{\prime}\left(\left\{ 0\right\} \right)=I_{0}\left(\Gamma^{\prime}\right)=I_{\lambda}\left(\Gamma\right)\,.
\end{align*}
as required.
\end{proof}

Now we can prove Theorem \ref{Theorem_regular}.

\begin{proof}[Proof of Theorem~\ref{Theorem_regular}]
Let $\lambda \in \sigma_{\mathfrak{p}}(\mathcal{H}_{\mathcal{T}})$ and assume, towards a contradiction, that all Q-Aomoto components are of the first type described in Observation~\ref{Observation_in_Q-Aomoto}, namely, finite trees consisting of vertices and their edges. Each such tree is connected to $\partial X_{\lambda}(\Gamma)$ via at least $d$ edges. However, all vertices in $\partial X_{\lambda}(\Gamma)$ also have degree $d$, and thus can absorb at most $d$ edges each from $X_{\lambda}\left(\Gamma\right)$. It follows that
\[
\left|\partial X_{\lambda}\left(\Gamma\right)\right|\geq\frac{d\cdot ccX_{\lambda}\left(\Gamma\right)}{d}=ccX_{\lambda}\left(\Gamma\right)
\]
which contradicts Theorem~\ref{Q-Aomoto_first_theorem}(iii).

In the case of the standard Schr\"odinger operator $\mathcal{H}=-\frac{d^{2}}{dx^{2}}$, this implies that
\[
\sigma_{\mathfrak{p}}\left(\mathcal{H}_{\mathcal{T}}\right)\subseteq\left\{ \frac{n^{2}\pi^{2}}{\ell_{e}^{2}}\mid e\in E,\ n\in\mathbb{N}\right\}.
\]
Indeed, for any $\lambda > 0$ to admit compactly supported eigenfunctions, it must satisfy $\sin(\sqrt{\lambda} \ell_e) = 0$ for some $e \in E$, i.e., $\sqrt{\lambda} \ell_e = n\pi$ for some $n \in \mathbb{N}$.
On the other hand, if $\lambda \leq 0$, then the corresponding solutions are either linear or exponential, and thus cannot vanish on both end points of any edge, implying the absence of such eigenfunctions.

\end{proof}

\begin{remark}\label{Remark_2_Regular}
In the case of a compact $2$-regular quantum graph---that is, a pure cycle---with $\alpha_{\upsilon} < \infty$ for every $\upsilon \in V$, each Q-Aomoto component (of either type) is connected to $\partial X_{\lambda}(\Gamma)$ via exactly two edges. In such a setting, an analogous argument to the one above yields that the universal cover has no eigenvalues and $\sigma_{\mathfrak{p}}\left(\mathcal{H}_{\mathcal{T}}\right)=\emptyset$.
This result is already known also from considerations of multiplicity
of eigenspaces on a line.
\end{remark}

\begin{proof}[Proof of Theorem~\ref{Q-Aomoto_second_theorem}]
This is a proof of part (ii), part (i) is obtained by setting $n=1$. Consider the corresponding discrete graph $\Gamma^{\prime}$ and let $G'$ be the discrete $n$-cover of $\Gamma'$ corresponding to $G$. Further let $\xi':G' \rightarrow \Gamma'$ be the covering map between the discrete graphs. Denote by $P$ and $P_{\partial}$ the orthogonal projections onto the subspaces of functions supported on $\xi'^{-1}\left(X_{0}\left(\Gamma^{\prime}\right)\right)$ and $\xi'^{-1}\left(\partial X_{0}\left(\Gamma^{\prime}\right)\right)$, respectively. We can restrict the Jacobi operator to $\operatorname{Im}(P)$, and obtain
\[
\operatorname{Im} \left(A_{G^{\prime}}|_{\operatorname{Im}(P)} \right) \subseteq \operatorname{Im}(P_{\partial})\,.
\]
Therefore, by the dimension theorem,
\begin{equation}
\dim\left(\ker\left(A_{G^{\prime}}|_{\operatorname{Im}(P)}\right)\right) \geq \dim\left(\mathbb{C}^{n \cdot |X_{0}(\Gamma^{\prime})|}\right) - \dim\left(\mathbb{C}^{n \cdot |\partial X_{0}(\Gamma^{\prime})|}\right) = n \cdot \left(ccX_{\lambda}(\Gamma) - |\partial X_{\lambda}(\Gamma)|\right) > 0\,. \label{dim_theorem}
\end{equation}

As in Notation~\ref{Notation_h_i_s}, let $T_{1}, \dots, T_{m}$ be the Q-Aomoto components of $G$, with corresponding vertices $t_{1}, \dots, t_{m} \in V(G^{\prime})$, and $h_{1}, \dots, h_{m}$ the normalized functions obtained from the projection of the eigenfunctions on these components. Then the mapping
\[
\ker\left(A_{G^{\prime}}|_{\operatorname{Im}(P)}\right) \hookrightarrow \ker\left(\mathcal{H}_{G} - \lambda\right)
\]
given by $\sum_{i=1}^{m} \alpha_{i} \delta_{t_{i}} \mapsto \sum_{i=1}^{m} \alpha_{i} h_{i}$ provides the desired subspace and completes the proof. 
\end{proof}
The inequality in this theorem may be strict. For example, in the compact regular graph of Example~\ref{example_regular}, one may impose the function $\cos(x)$ on each edge (the orientation is irrelevant). This function is indeed an eigenfunction of the Schr\"odinger operator on the compact graph; however, it cannot be generated by the above construction, since it does not vanish at every vertex. By Theorem~\ref{Q-Aomoto_first_theorem}(i), this implies that the function does not vanish outside the lifted Q-Aomoto set.

\begin{proof}[Proof of Theorem~\ref{Q-Aomoto_third_theorem}]
We construct the corresponding discrete graph $\Gamma_{X}^{\prime}$ in the same manner as for the Q-Aomoto set, but using the connected components of $X$ instead of $X_{\lambda}(\Gamma)$. The vertices $t_{1}, \dots, t_{n}$ correspond to the components $T_{1}, \dots, T_{n} \subseteq X$.

From Lemma~\ref{Lemma_From_Derived_Graph}(ii) (with Dirichlet boundary conditions imposed on $\partial X$), the eigenspaces associated with each connected component of $X$ are one-dimensional. We denote by $h_{1}, \dots, h_{n}$ the corresponding normalized real eigenfunctions on these components.

As in the proof of Theorem~\ref{Q-Aomoto_second_theorem}, the dimension theorem implies that
\[
\dim\left(\ker\left(A_{\Gamma_{X}^{\prime}}\right)\right) \geq cc(X) - |\partial X|\,.
\]
This yields Theorem~\ref{Q-Aomoto_third_theorem}(i), via the isometric inclusion
\[
\sum_{i=1}^{m} \alpha_{i} \delta_{t_{i}} \mapsto \sum_{i=1}^{m} \sum_{s \in S_{i}} \alpha_{i} h_{i}\,.
\]

For Theorem \ref{Q-Aomoto_third_theorem}(ii), let $\Gamma_{n}$ be a sequence of $m_{n}$-covers of $\Gamma$
satisfying the assumptions of Theorem~\ref{ESM_limit}(i), and let $\nu_{n}$ denote their normalized eigenvalue counting measures. Let $\mu$ be the density of states of $\Gamma$.

From Theorem~\ref{Q-Aomoto_third_theorem}(i), we obtain the lower bound:
\[
\nu_{n}\left(\left\{ \lambda\right\} \right)=\frac{\dim\left(\ker\left(\mathcal{H}_{\Gamma_{n}}-\lambda\right)\right)}{\mathcal{L}_{E_{n}}}\geq\frac{m_{n}\left(cc\left(X\right)-\left|\partial X\right|\right)}{m_{n}\mathcal{L}_{E}}=\frac{cc\left(X\right)-\left|\partial X\right|}{\mathcal{L}_{E}}\,.
\]
Therefore, by Theorem~\ref{ESM_limit}(ii) and the regularity of $\mu$ (established in Corrolary~\ref{DOS_regular}), it follows that
\[
\mu\left(\left\{ \lambda\right\} \right)\geq\frac{cc\left(X\right)-\left|\partial X\right|}{\mathcal{L}_{E}}>0\,.
\]
\end{proof}

The proof of Theorem~\ref{Berkolaiko_theorem} is based on the following result of Berkolaiko and Liu~\cite{Berkolaiko}.

\begin{lemma}\label{Lemma_Berkolaiko} \cite[Theorem 3.6]{Berkolaiko}
Let \( \Gamma \) be a connected compact quantum graph equipped with the standard Schr\"odinger operator \( \mathcal{H} = -\frac{d^{2}}{dx^{2}} \), and let all vertices satisfy \(\delta\)-type conditions, where Dirichlet conditions are permitted only at vertices of degree one. Suppose \( \Gamma \) is not a pure cycle. Then, for a residual set of edge lengths in \( \mathbb{R}_{+}^{|E|} \), every eigenfunction \( f \) of \( \mathcal{H} \) satisfies one of the following:
\begin{enumerate}[label=\roman*.]
\item \( f(\upsilon) \neq 0 \) for every vertex \( \upsilon \) with \( \alpha_{\upsilon} < \infty \).
\item \( \mathrm{supp}(f) = L \), where \( L \) is a single pure cycle in \( \Gamma \).
\end{enumerate}
\end{lemma}

\begin{proof}[Proof of Theorem~\ref{Berkolaiko_theorem}]
If $\Gamma$ is a pure cycle, then it is a compact $2$-regular quantum graph. As established in Remark~\ref{Remark_2_Regular}, in such a case the universal cover does not admit eigenvalues.

Assume now that $\Gamma$ is not a pure cycle and assume first that $\Gamma$ does not contain a pure cycle. By Theorem~\ref{Q-Aomoto_second_theorem}(i), $\sigma_{\mathfrak{p}}\left(\mathcal{H}_{\mathcal{T}}\right) \subseteq  \sigma_{\mathfrak{p}}\left(\mathcal{H}_{\Gamma}\right)$ and for each $\lambda$-eigenfunction, $\tilde{f}$, on $\mathcal{T}$ there exists a $\lambda$-eigenfunction, $f$, on $\Gamma$, supported on \( \Xi\left(\mathrm{supp}(\tilde{f})\right) \). Furthermore, for any \( \upsilon \in V \) with \( f(\upsilon) \neq 0 \), there exists \( \tilde{\upsilon} \in \Xi^{-1}(\upsilon) \) such that \( \tilde{f}(\tilde{\upsilon}) \neq 0 \). By Lemma \ref{Lemma_Berkolaiko}, for a residual set of edge-lengths these are all vertices $\upsilon \in V(\Gamma)$ with $\alpha_\upsilon<\infty$, which is a contradiction to the fact that $\Gamma$ contains a cycle. Thus, for a residual set of edge-lengths, $\sigma_{\mathfrak{p}}\left(\mathcal{H}_{\mathcal{T}}\right)$ is empty.

We may therefore restrict our attention to finite graphs that are not a pure cycle but contain a pure cycle. So let $L\subseteq\Gamma$ be a pure cycle that is connected to the rest of the graph via a vertex $u \in V$. We construct a modified graph by replacing $u$ with $\left|E_{u}\right|$ distinct vertices, where each new vertex is connected to a single edge from $E_{u}$, and we impose Dirichlet boundary conditions at each of these new vertices.

As a result, for every such cycle $L$, we obtain a finite collection of connected subgraphs $\Gamma_{1}^{L}, \dots, \Gamma_{m_{L}}^{L}, \widehat{L}$, where $\widehat{L}$ is the modification of $L$ after this process. An illustration of this decomposition is provided in Figure~\ref{fig:Sepetetion_Berkolaiko}.

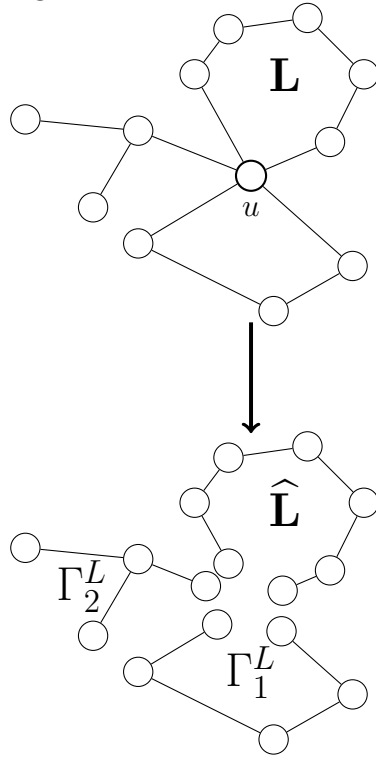
\begin{minipage}{0.9\textwidth}
\begin{center}
\begin{tikzpicture}[scale=1.5]

\begin{scope}[shift={(0,4)}]

\node[circle, draw, thick, minimum size=4mm, label=below:$u$] (u) at (0,0) {};

\node[circle, draw, minimum size=3mm] (x1) at (-1,0.4) {};
\node[circle, draw, minimum size=3mm] (x2) at (-1,-0.6) {};
\node[circle, draw, minimum size=3mm] (x3) at (0.9,-0.8) {};

\draw (u) -- (x1);
\draw (u) -- (x2);
\draw (u) -- (x3);

\node[circle, draw, minimum size=3mm] (y1) at (-1.4,-0.28) {};
\node[circle, draw, minimum size=3mm] (z1) at (-2,0.7-0.2) {};
\draw (x1) -- (y1);
\draw (x1) -- (z1);

\node[circle, draw, minimum size=3mm] (y2) at (0.2,-1.2) {};
\draw (x2) -- (y2) -- (x3);

\node[circle, draw, minimum size=3mm] (a) at (0.7,0.3) {};
\node[circle, draw, minimum size=3mm] (b) at (1.0,0.9) {};
\node[circle, draw, minimum size=3mm] (c) at (0.5,1.4) {};
\node[circle, draw, minimum size=3mm] (d) at (-0.2,1.3) {};
\node[circle, draw, minimum size=3mm] (e) at (-0.5,0.9) {};

\draw (u) -- (a) -- (b) -- (c) -- (d) -- (e) -- (u);


\node at (0.3,0.9) {\fontsize{18}{21}\selectfont \textbf{L}};
\end{scope}

\begin{scope}[shift={(0,0.2)}]
\node[circle, draw, minimum size=3mm] (u1) at (-1*0.4,0.4*0.4) {};
\node[circle, draw, minimum size=3mm] (u2) at (-1*0.3,-0.6*0.3) {};
\node[circle, draw, minimum size=3mm] (u3) at (0.9*0.3,-0.8*0.3) {};
\node[circle, draw, minimum size=3mm] (u4) at (0.7*0.4,0.3*0.4) {};
\node[circle, draw, minimum size=3mm] (u5) at (-0.5*0.4,0.9*0.4) {};

\node[circle, draw, minimum size=3mm] (x1) at (-1,0.6-0.2) {};
\node[circle, draw, minimum size=3mm] (x2) at (-1,-0.6) {};
\node[circle, draw, minimum size=3mm] (x3) at (0.9,-0.8) {};

\draw (u1) -- (x1);
\draw (u2) -- (x2);
\draw (u3) -- (x3);

\node[circle, draw, minimum size=3mm] (y1) at (-1.4,-0.28) {};
\node[circle, draw, minimum size=3mm] (z1) at (-2,0.7-0.2) {};
\draw (x1) -- (y1);
\draw (x1) -- (z1);

\node[circle, draw, minimum size=3mm] (y2) at (0.2,-1.2) {};
\draw (x2) -- (y2) -- (x3);

\node[circle, draw, minimum size=3mm] (a) at (0.7,0.3) {};
\node[circle, draw, minimum size=3mm] (b) at (1.0,0.9) {};
\node[circle, draw, minimum size=3mm] (c) at (0.5,1.4) {};
\node[circle, draw, minimum size=3mm] (d) at (-0.2,1.3) {};
\node[circle, draw, minimum size=3mm] (e) at (-0.5,0.9) {};

\draw (u4) -- (a) -- (b) -- (c) -- (d) -- (e) -- (u5);

\node at (0.3,0.9) {\fontsize{18}{21}\selectfont $\widehat{\textbf{L}}$};
\node at (0,-0.6) {\fontsize{18}{21}\selectfont\textbf{$\Gamma_{1}^{L}$}};
\node at (-1.5,0.16) {\fontsize{18}{21}\selectfont\textbf{$\Gamma_{2}^{L}$}};
\end{scope}

\draw[->, thick, line width=1.5pt, shorten >=5pt] (0,2.7) -- (0,1.6);

\end{tikzpicture}
\end{center}
\end{minipage}
\vspace{1em}
\captionof{figure}{The top graph contains a loop $L$ connected to the rest of the graph via a vertex $u$ of degree $5$. In the bottom graph, the vertex $u$ is split into five distinct vertices of degree $1$, each retaining one of the incident edges. This operation separates $L$ from the rest of the graph, and the remaining part of the graph is also partitioned into two connected components, denoted by $\Gamma_{1}^{L}$ and $\Gamma_{2}^{L}$.}
\label{fig:Sepetetion_Berkolaiko}

\vspace{1em}
We aim to prove that for each \( i \), $\sigma_{\mathfrak{p}}\!\left(\mathcal{H}_{\Gamma_{i}^{L}}\right) \cap \sigma_{\mathfrak{p}}\!\left(\mathcal{H}_{\widehat{L}}\right) = \emptyset$ for a residual set of edge lengths. To this end, it suffices to show that \( 0 \notin \sigma_{\mathfrak{p}}(\mathcal{H}_{\widehat{L}}) \) on an open and dense set of lengths, and that for every \( n \in \mathbb{N} \), the \( n \)-th eigenvalue of \( \mathcal{H}_{\Gamma_{i}^{L}} \) (see Lemma~\ref{Eigenvelues_to_Infinity}) is not an eigenvalue of \( \mathcal{H}_{\widehat{L}} \) on an open and dense set. The intersection of all these sets is then residual.

To see that $0 \notin \sigma_{\mathfrak{p}}\!\left(\mathcal{H}_{\widehat{L}}\right)$ note that an eigenfunction with this value is linear on each edge. By the Dirichlet conditions at the two copies of $u$ in $\widehat{L}$, it has to vanish at both these two vertices. Fixing an edge, $e_1$, emanating from $u$, for any choice of all other edge-lengths in $\widehat{L}$, there is at most one choice of length for $e_1$ that will make $0$ an eigenvalue. Thus, the set of edge-lengths in $\widehat{L}$ for which $0 \notin \sigma_{\mathfrak{p}}\!\left(\mathcal{H}_{\widehat{L}}\right)$ is dense. Moreover, since eigenvalues depend continuously on the edge lengths
(see, e.g., \cite{berkolaiko8,Introduction_to_Quantum_Graphs}),
small perturbations preserve strict inequalities between eigenvalues and zero, so this set is also open.

Now let \( n \in \mathbb{N} \) be such that the \( n \)-th eigenvalue of
\( \mathcal{H}_{\Gamma_{i}^{L}} \) is nonzero.
According to \cite{Berkolaiko}, the nonzero eigenvalues of \( \mathcal{H}_{\widehat{L}} \)
can be controlled by varying the length of a single edge in \( \widehat{L} \).
Therefore, for fixed edge-lengths of \( \Gamma_{i}^{L} \) and all but one edge-lengths in $\widehat{L}$, the set of lengths of the excluded edge, for which this eigenvalue coincides with an eigenvalue of \( \mathcal{H}_{L} \) is countable.
Its complement is thus dense, and again openness follows from continuity of eigenvalues with respect to edge lengths.

Taking the intersection over all pure cycles \( L \subseteq \Gamma \) and all corresponding components \( \Gamma_{i}^{L} \), and intersecting further with the residual set obtained in Lemma~\ref{Lemma_Berkolaiko}, we obtain a residual set
\( S \subseteq \mathbb{R}_{+}^{|E|} \)
of edge lengths such that, for every pure cycle \( L \) and every \( i \),
\[
\sigma_{\mathfrak{p}}\!\left(\mathcal{H}_{\Gamma_{i}^{L}}\right)
\cap
\sigma_{\mathfrak{p}}\!\left(\mathcal{H}_{\widehat{L}}\right)
= \emptyset .
\]

\medskip

Let \( \ell \in S \), and assume that for the edge lengths \( \ell \), there exists $\lambda \in \sigma_{\mathfrak{p}}\left(\mathcal{H}_{\mathcal{T}}\right)$ with an eigenfunction \( \tilde{f} \). Then, again by Theorem~\ref{Q-Aomoto_second_theorem}(i), $\Xi\left(\mathrm{supp}\left(\tilde{f} \right)\right)$ is contained in the $\lambda$-Q Aomoto set, $\lambda \in \sigma_{\mathfrak{p}}\left(\mathcal{H}_{\Gamma}\right)$, and there exists an eigenfunction \( f \) supported on \( \Xi\left(\mathrm{supp}\left(\tilde{f} \right)\right) \). Furthermore, for any \( \upsilon \in V \) with \( f(\upsilon) \neq 0 \), there exists \( \tilde{\upsilon} \in \Xi^{-1}(\upsilon) \) such that \( \tilde{f}(\tilde{\upsilon}) \neq 0 \).

If \( f(\upsilon) \neq 0 \) for all \( \upsilon \in V \) with \( \alpha_{\upsilon} < \infty \), then the Q-Aomoto set of \( \Gamma \) contains a cycle (since $\alpha_\upsilon=\infty$ only on leaves), contradicting Theorem~\ref{Q-Aomoto_first_theorem}(i). Therefore, by Lemma~\ref{Lemma_Berkolaiko}, we must have
\[
\mathrm{supp}(f) = L
\]
for some unique pure cycle \( L \subseteq \Gamma \), and the vertex \( u \) connecting \( L \) to the rest of the graph must lie in \( \partial X_{\lambda}(\Gamma) \) (by Lemma~\ref{Remark_Pure_Cycle}). It follows that $\lambda \in \sigma_{\mathfrak{p}}\!\left(\mathcal{H}_{\widehat{L}}\right)$.

Now, if \( X_{\lambda}(\Gamma) \subseteq L \), then by the properties of cycles we have
\( I_{\lambda}(\Gamma)=0 \), which yields a contradiction. Therefore, there exists
\( i \in [m_{L}] \) such that
\[
X_{\lambda}(\Gamma)\cap\Gamma_{i}^{L}\neq\emptyset .
\]
Hence, there exists an edge \( e\in X_{\lambda}(\Gamma)\cap\Gamma_{i}^{L} \),
a lift \( \widetilde{e}\in\Xi^{-1}(e) \), and a function
\( g\in\ker(\lambda-\mathcal{H}_{\mathcal{T}}) \)
such that \( g|_{\widetilde{e}} \) is not identically zero and
\[
g(\widetilde{u})=0
\qquad
\text{for every } \widetilde{u}\in\Xi^{-1}(u).
\]

We now construct a modified graph from \( \mathcal{T} \) as follows.
Replace each vertex in \( \Xi^{-1}( u) \) by \( |E_{u}| \) distinct vertices, where each new vertex
is incident to exactly one edge from \( E_{u} \).
Impose Dirichlet boundary conditions at each of these new vertices, and let
\( \mathcal{T}_{i}^{L} \) denote the connected component containing \( \widetilde{e} \).
It is straightforward to verify that \( \mathcal{T}_{i}^{L} \) is the universal cover
of \( \Gamma_{i}^{L} \).
Moreover, the restriction \( g|_{\mathcal{T}_{i}^{L}} \) defines an eigenfunction, and therefore
\[
\lambda \in \sigma_{\mathfrak{p}}\!\left(\mathcal{H}_{\mathcal{T}_{i}^{L}}\right).
\]

By Theorem~\ref{Q-Aomoto_second_theorem}(i),
\[
\sigma_{\mathfrak{p}}\!\left(\mathcal{H}_{\mathcal{T}_{i}^{L}}\right)
\subseteq
\sigma_{\mathfrak{p}}\!\left(\mathcal{H}_{\Gamma_{i}^{L}}\right).
\]
Consequently,
\[
\lambda \in
\sigma_{\mathfrak{p}}\!\left(\mathcal{H}_{\Gamma_{i}^{L}}\right)
\cap
\sigma_{\mathfrak{p}}\!\left(\mathcal{H}_{\widehat{L}}\right),
\]
which contradicts the choice of edge lengths. This completes the proof.
\end{proof}

\end{document}